\documentclass[11pt]{article}
\usepackage{amsmath, amsthm,amssymb,latexsym,mathscinet}
\usepackage{enumerate, color}
\usepackage{caption,subcaption,xcolor}
\usepackage{tikz,graphicx,cite}
\usepackage{todonotes} 

\numberwithin{equation}{section}

\newtheorem{theorem}{Theorem}[section]
\newtheorem{lemma}[theorem]{Lemma}
\newtheorem{claim}{Claim}[theorem]

\newtheorem{corollary}[theorem]{Corollary}
\newtheorem{construction}[theorem]{Construction}

\newtheorem{question}{Question}

\newcommand{\mad}{\mathrm{mad}}

\newcommand{\FII}{{$\mathcal{F}\mathcal{I}\mathcal{I}$-partition}}

\newcommand{\FIIe}{$\mathcal{F}\mathcal{I}\mathcal{I}$-\it{partition}}
\newcommand{\fii}[3]{${#1}\sqcup{#2}\sqcup{#3}$}
\newcommand{\fc}[1]{\mu^*({#1})}
\newcommand{\rc}[1]{{\bf {[C\ref{#1}]}}}
\newcommand{\rcp}[1]{{\bf {[C${}^{\prime}$\ref{#1}]}}}
\newcommand{\ra}[1]{{\bf {[P\ref{#1}]}}}

\title{\bf  On Star 5-Colorings of Sparse Graphs}
\author{Ilkyoo Choi\footnote{Ilkyoo Choi was supported by the Basic Science Research Program through the National Research Foundation of Korea (NRF) funded by the Ministry of Education (NRF-2018R1D1A1B07043049), and also by the Hankuk University of Foreign Studies Research Fund.}\\[1.5ex]
\small Department of Mathematics\\
\small Hankuk University of Foreign Studies  \\
\small Yongin, Republic of Korea \\
\small\tt ilkyoo@hufs.ac.kr \\
\and
Boram Park\footnote{Boram Park work  supported by Basic Science Research Program through the National Research Foundation of Korea (NRF) funded by the Ministry of Science, ICT and Future Planning (NRF-2018R1C1B6003577).} \\[1.5ex]
\small Department of Mathematics\\
\small Ajou University\\
\small Suwon, Republic of Korea \\
\small\tt borampark@ajou.ac.kr}

\begin{document}

\maketitle
\date

\begin{abstract} \normalsize
A \textit{star $k$-coloring} of a graph $G$ is a proper (vertex) $k$-coloring of $G$ such that the vertices on a path of length three receive at least three colors.
Given a graph $G$, its \textit{star chromatic number}, denoted $\chi_s(G)$, is the minimum integer $k$ for which $G$ admits a star $k$-coloring.
Studying star coloring of sparse graphs is an active area of research, especially in terms of the maximum average degree of a graph; the \textit{maximum average degree}, denoted $\mad(G)$, of a graph $G$ is $\max\left\{ \frac{2|E(H)|}{|V(H)|}:{H \subset G}\right\}$.
It is known that for a graph $G$,
if $\mad(G)<\frac{8}{3}$, then $\chi_s(G)\leq 6$~\cite{2010KuTi}, and
if $\mad(G)< \frac{18}{7}$ and its girth is at least 6, then $\chi_s(G)\le 5$~\cite{2009BuCrMoRaWa}.
We improve both results by showing that for a graph $G$, if $\mad(G)\le \frac{8}{3}$, then $\chi_s(G)\le 5$.
As an immediate corollary, we obtain that a planar graph with girth at least 8 has a star 5-coloring, improving the best known girth condition for a planar graph to have a star 5-coloring~\cite{2010KuTi,2008Ti}.
\end{abstract}

\section{Introduction}

All graphs in this paper are simple.
Given a graph $G$, a {\it proper $k$-coloring} of $G$ is a partition of its vertex set $V(G)$ into $k$ parts such that there is no edge with both endpoints in the same part.
In other words, each color class induces an empty graph.

As a generalization of proper coloring, Gr\"unbaum~\cite{1973Gr} introduced the notion of  \textit{acyclic coloring}, which is a proper coloring satisfying the additional constraint that the vertices on a cycle (of any length) receive at least three colors.
In other words, the union of two color classes induces an acyclic graph.
One of the most interesting results regarding acyclic coloring is the result by Borodin~\cite{1979Bo}, which states that every planar graph admits an acyclic coloring with five colors.
This resolved a conjecture in the initial paper~\cite{1973Gr} of Gr\"unbaum where he showed that five colors is necessary to acyclically color certain planar graphs; Kostochka and Mel\cprime nikov~\cite{1976KoMe} even constructed a bipartite planar graph requiring five colors when acyclically colored.
In contrast, the famous Four Color Theorem~\cite{1977ApHa,1977ApHaKo} states that every planar graph has a proper 4-coloring.

In~\cite{1973Gr}, Gr\"unbaum also raised the question of proper coloring with the additional constraint that the vertices on a path of length three receive at least three colors.
In other words, the union of two color classes induces a star forest.
Although Gr\"unbaum gave no specific name for this type of coloring, this coloring is now known as {\it star coloring}, ever since the term was first coined by Albertson et al.~\cite{2004AlChKiKuRa}.
To be precise, a {\it star $k$-coloring} of a graph $G$ is a proper $k$-coloring of $G$ where the vertices on a path of length three receive at least three colors.
The {\it star chromatic number} of a graph $G$, denoted $\chi_s(G)$, is the minimum $k$ for which $G$ admits a star $k$-coloring.
Since a star forest is also an acyclic graph, star coloring is a strengthening of acyclic coloring.
Acyclic coloring and star coloring have been an active area of research, and we direct the readers to a thorough survey by Borodin~\cite{2013Bo} for the rich literature.
There is also an edge-coloring analogue for star coloring; for recent progress on star edge-coloring subcubic graphs, see~\cite{2013DvMoSa,2018KeRa,lei2017star,luvzar2017star}.

In this paper, we are interested in star colorings of sparse graphs, where sparsity is measured in terms of the maximum average degree.
The {\em maximum average degree} of a graph $G$, denoted $\mad(G)$, is the maximum of the average degrees of all its subgraphs, that is, ${\displaystyle \mad(G) = \max\left\{ {\frac{2|E(H)|}{|V(H)|}}:H\subset G\right\}}$.
Since a planar graph $G$ with girth at least $g$ satisfies $\mad(G)<\frac{2g}{g-2}$, a result regarding graphs with bounded maximum average degree implies that planar graphs with certain girth conditions can reach the same conclusion,  see \cite{DanDoug}.

Gr\"unbaum proved that planar graphs are star 2304-colorable in~\cite{1973Gr} back in 1973, and after 45 years the best result so far is by Albertson et al.~\cite{2004AlChKiKuRa} where they showed that all planar graphs are star 20-colorable.
They also constructed a planar graph that requires at least ten colors to be star colored, and for a given girth $g$, they constructed a planar graph with girth $g$ that requires at least four colors to be star colored.
Moreover, they investigated the star chromatic number for planar graphs with certain girth constraints, where they proved that a planar graph $G$  with girth at least 5 and 7 satisfies $\chi_s(G)\leq 16$ and $\chi_s(G)\leq 9$, respectively, improving upon some bounds in \cite{2004FeRaRe}.
Timmons~\cite{2008Ti} and K\"undgen and Timmons~\cite{2010KuTi} continued the study as they obtained results that imply a planar graph is star $4$-, $5$-, $6$-, $7$-, $8$-colorable if its girth is at least 14, 9, 8, 7, 6, respectively.
Sufficient conditions on girth to guarantee that a planar graph is star $4$-colorable have received much attention due to its relation to the Four Color Theorem~\cite{1977ApHa,1977ApHaKo}.
In particular, Bu et al.~\cite{2009BuCrMoRaWa} improved the girth constraint to 13, and Brandt et al.~\cite{2016BrFeKuLoStYa} has the current best bound showing that a planar graph with girth at least 10 has a star 4-coloring.

Lower bounds on the girth constraints have also been investigated.
In particular, a planar graph with girth 7 and 5 that requires 5 and 6 colors to be star colored has been constructed in~\cite{2008Ti} and~\cite{2010KuTi}, respectively.
See Table~\ref{table} for a summary of lower and upper bounds on the star chromatic number of a planar graph with a given girth constraint.

\begin{table}[ht]
\centering	
\begin{tabular}{c||c|c|c|c|c|c|c|c|c|c}
\mbox{girth} & 3& 4& 5& 6& 7& 8& 9& $\geq 10$\\
\hline
\hline
\mbox{upper bound} & 20~\cite{2004AlChKiKuRa}& 20~\cite{2004AlChKiKuRa} & 16~\cite{2004AlChKiKuRa}& 8~\cite{2010KuTi} & 7~\cite{2010KuTi}& 6~\cite{2010KuTi}& 5~\cite{2008Ti} & 4~\cite{2016BrFeKuLoStYa} \\
\mbox{lower bound}
&           	 10~\cite{2004AlChKiKuRa} 			
&			 10~\cite{2004AlChKiKuRa} 			
& 			6~\cite{2010KuTi}					
& 			5~\cite{2008Ti}						
&			 5~\cite{2008Ti}						
& 			4~\cite{2004AlChKiKuRa} 			
& 			4~\cite{2004AlChKiKuRa}				
& 		4~\cite{2004AlChKiKuRa}				
\end{tabular}
\caption{Table of best known results.}
\label{table}
\end{table}

Various results above are also true for the maximum average degree setting~\cite{2016BrFeKuLoStYa,2010KuTi,2008Ti} and the list version setting~\cite{2009BuCrMoRaWa,2011ChRaWa,2014ChRaWa,2010KuTi}.
Researchers have also investigated star coloring for bipartite planar graphs~\cite{2009KiKuTi} and subcubic graphs~\cite{2013ChRaWa,2014ChRaWa}.
In particular, we explicitly state the following two theorems:

\begin{theorem}[\cite{2009BuCrMoRaWa}]\label{thm1.1}
For a graph $G$,
if $\mad(G)< \frac{18}{7}$ and its girth is at least 6, then $G$ is star 5-colorable.
\end{theorem}
\begin{theorem}[\cite{2010KuTi}]\label{thm1.2}
For a graph $G$,
if $\mad(G)< \frac{8}{3}$, then $G$ is star 6-colorable.
\end{theorem}

Our main theorem improves both aforementioned theorems.
Note that Theorems~\ref{thm1.1} and~\ref{thm1.2} imply that a planar graph with girth at least 8 and 9 has a star coloring with 6 and 5 colors, respectively.
Likewise, as a direct consequence of our main result, we improve the best known girth condition for a planar graph to be star 5-colorable.
We now present our main result and its direct consequence:

\begin{theorem}\label{thm:main:star-coloring}
For a graph $G$,
if $\mad(G)\le \frac{8}{3}$, then $G$ is star 5-colorable.
\end{theorem}

\begin{corollary}
A planar graph with girth at least 8 is star 5-colorable.
\end{corollary}

We actually prove a stronger statement, guaranteeing a certain partition of the vertices that implies the existence of a star 5-coloring.
For a positive integer $k$, a {\it $k$-independent set} of a graph $G$ is a subset $S$ of $V(G)$ such that a pair of vertices in $S$ has distance at least $k+1$ in $G$.
For two disjoint sets $A$ and $B$, let $A \sqcup B$ denote the disjoint union of $A$ and $B$.
We say a graph $G$ has an
{\FIIe}
 \fii{F}{I_{\alpha}}{I_{\beta}},  if $F, I_{\alpha}, I_{\beta}$ is a partition of $V(G)$, each of $I_{\alpha}$ and  $I_{\beta}$ induces a 2-independent set, and $F$ induces a forest.
Since a forest is star 3-colorable (by picking a root and coloring the vertices according to the distance to the root modulo three), if a graph $G$ has an {\FII} \fii{F}{I_{\alpha}}{I_{\beta}}, then $G$ is star 5-colorable; use three colors on $F$, one color on $I_\alpha$, and one color on $I_\beta$.
The above idea of using a 2-independent set first appeared in Albertson et al. \cite{2004AlChKiKuRa}.
Hence, in order to prove Theorem~\ref{thm:main:star-coloring}, it is sufficient to show the following  Theorem~\ref{thm:main:8/3}.

\begin{theorem}\label{thm:main:8/3}
For a graph $G$, if $\mad(G)\le \frac{8}{3}$, then  $G$ has an {\FII}.
\end{theorem}

The paper is organized as follows. The proof of Theorem~\ref{thm:main:8/3} is split into  Sections~\ref{sec:outline:main}, \ref{sec:proof:configuration},~and~\ref{sec:proof:configuration2}. Section~\ref{sec:outline:main} lays out the discharging rules and reducible configurations. Sections~\ref{sec:proof:configuration} and~\ref{sec:proof:configuration2} provide
the proofs of the reducible configurations.
We conclude with  questions and tightness bounds in Section~\ref{sec:remark}.

We list some important definitions used in this paper.
We use $[n]$ to denote the set $\{1, \ldots, n\}$.
Let $G$ be a graph.
For  $S\subset V(G)$, let $G-S$ denote the graph obtained from $G$ by deleting the vertices in $S$.
If $S=\{x\}$, then we denote $G-S$ by $G-x$.
Likewise, in order to improve the readability of the paper, we often drop the braces and commas to denote a set and use `$+$' for the set operation `$\cup$'.
For instance, given $A\subset V(G)$ and $x,y,z\in V(G)$, we use $A+x-y$  and $A-z+xy$ to denote $(A\cup\{x\})\setminus\{y\}$  and $(A\setminus\{z\})\cup\{x,y\}$, respectively.

A {\it $d^+$-vertex}, {\it $d$-vertex}, {\it $d^-$-vertex} is a vertex of degree at least $d$, exactly $d$, at most $d$, respectively.
Given a vertex $x$, a neighbor of $x$ with degree at least $d$, exactly $d$, at most $d$ is called a {\it $d^+$-neighbor}, {\it $d$-neighbor}, {\it $d^-$-neighbor}, respectively.
For $S\subset V(G)$, a vertex in a set $S$ is called an {\it $S$-vertex}. Similarly, we say $u$ is an $S$-{\it neighbor} of a vertex $v$ if $u\in N_G(v)\cap S$.
A {\it pendent $k$-cycle} is a  cycle of length $k$ where all its vertices except one vertex $x$ are $2$-vertices;
we also say this cycle is {\it at the vertex $x$}.
A $3$-cycle is also called a \textit{triangle}.

We finish this section with observations, which is frequently used in the proof.

\begin{lemma}\label{lema:2B2}
Let  $H$ be a graph with an {\FII} \fii{F}{I_{\alpha}}{I_{\beta}}.
\begin{itemize}
    \item[\rm(i)] If $H$ has an induced subgraph isomorphic to $J_1$ in Figure~\ref{fig:2B_2} and $v^*\in I_{\alpha}$, then $w_1\in I_{\beta}$.
        \item[\rm(ii)] If $H$ has an induced subgraph  isomorphic to $J_2$ in Figure~\ref{fig:2B_2}, then $v^*\in F$.
\end{itemize}
\end{lemma}

\begin{figure}[ht]
	\centering
  \includegraphics[scale=0.75,page=1]{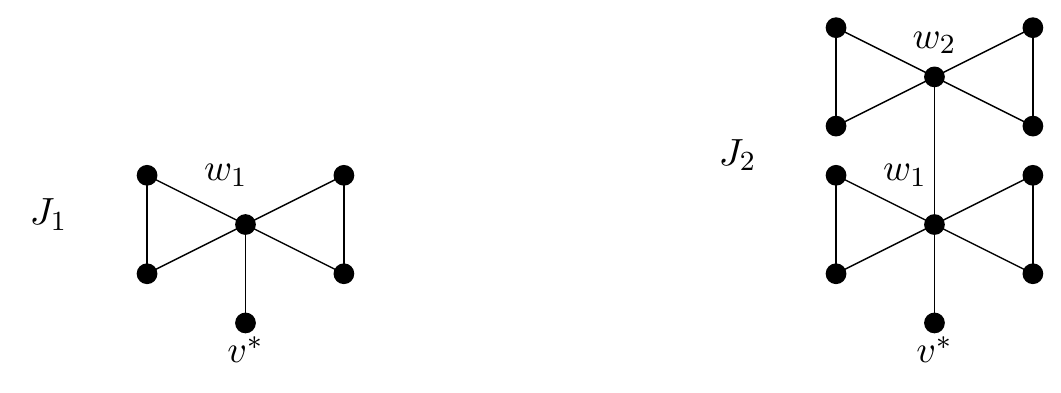}
  \caption{The graphs $J_1$ and $J_2$.}
  \label{fig:2B_2}
\end{figure}

\begin{proof}
To show (i), suppose to the contrary that $w_1\not\in I_\beta$. Then  $w_1\in F$, and so one vertex from each pendent triangle at $w_1$ is in $I_\beta$, which is a contradiction since $I_\beta$ is a 2-independent set.
To show (ii), suppose to the contrary that $v^*\not\in F$, say $I_{\alpha}$.
By (i), $w_1\in I_\beta$, and again by (i), $w_2\in I_\alpha$, which is a contradiction since $I_\alpha$ is a 2-independent set.
\end{proof}

In order to check the maximum average degree of a graph, we often use the potential function.
For a graph $H$, let $\rho_H:2^{V(H)} \rightarrow \mathbb{Z}$ be the function such that $\rho_H(A)=4|A|-3|E(H[A])|$ for $A\subset V(H)$, called the \textit{potential function} for $H$.
We use the following, which is straightforward from the definition:
\[ \qquad \rho_H(A)\ge 0\text{ for all }A\subset V(H)\text{ if and only if }\mad(H) \le \frac{8}{3}.\]
For $I\subset V(H)$, let \[  \rho^*_H(I)=\min\{ \rho_H(K)\mid I\subset K \subset V(H) \} .\]
For brevity, we often drop the braces and commas to denote $\rho^*_H(A)$, such as $\rho^*_H(ab)$ instead of $\rho_H^*(\{a,b\})$.
An easy counting argument shows that for subsets $A$ and $B$ of $V(H)$,
$\rho_H(A)+\rho_H(B) \ge \rho_H(A\cup B)+\rho_H(A\cap B)$.
This further implies the following:
\begin{eqnarray}\label{eqref:basic:rho}
&&\rho^*_H(A)+\rho^*_H(B) \ge \rho^*_H(A\cup B)+\rho^*_H(A\cap B).
\end{eqnarray}
For a graph $H$ with $\mad(H)\le \frac{8}{3}$ and disjoint subsets $S,T$ of $V(H)$, if $T$ contains all vertices not in $S$ that are adjacent to an $S$-vertex, then since $\rho_{H}(S\cup T)\ge 0$,
we have
\begin{eqnarray}\label{eqref:basic:rho2}
&&
\rho^*_{H-S}(T) \ge -4\cdot |S|+3\cdot
(\text{the number of edges in }H\text{ incident with an }S\text{-vertex}).
\end{eqnarray}
Finally, for graphs $H$ and $J$ with $\mad(H), \mad(J)\le \frac{8}{3}$,
let $H'$ be the graph obtained from $H$ by attaching $J$ to a vertex,
that is, identifying a vertex $v$ of $H$ and a vertex $w$ of $J$.
If $\rho^*_H(v)\ge k\ge 0$ and $\rho_J^*(w)\ge 4-k$, then attaching $J$ decreases the potential by at most $k$, and so $\mad(H')\le \frac{8}{3}$.
Throughout the paper, we often attach a pendent triangle, $J_1$, and $J_2$, which decreases the potential by at most 1, 1, and 2, respectively.

\section{Discharging Procedure}\label{sec:outline:main}

Throughout the figures in the paper, the degree of a solid (black) vertex is the number of incident edges drawn in the figure, whereas a hollow (white) vertex means a $2^+$-vertex. For a graph $H$, let $V^*(H)$ be the set of vertices of $H$ except the 2-vertices on a pendent cycle.
Suppose to the contrary that a counterexample $G$ exists to Theorem~\ref{thm:main:8/3};
namely, $\mad(G)\le \frac{8}{3}$ but $G$ has no {\FII}.
Choose $G$ to be a minimum counterexample with respect to
\begin{quote}
 (1) $|V^*(G)|$ is minimum, \quad (2) $|V(G)|$ is minimum.
\end{quote}

We provide a list of subgraphs where each subgraph does not appear in $G$; each subgraph is also referred to as a {\it reducible configuration}.
We first define the following sets, see Figure~\ref{fig:Wsets}.
\begin{eqnarray*}
W_2&=&\{x\in V(G) :x\text{ is a }2\text{-vertex not on a pendent triangle} \}\\
W_3&=&\{x\in V(G) : x\text{ is a }3\text{-vertex with  two 2-neighbors} \}\\
W_4&=&\{x\in V(G) :x\text{ is a }4\text{-vertex on one pendent triangle} \}\\
W_5&=&\{x\in V(G): x\text{ is a }5\text{-vertex on two pendent triangles} \}\\
V_k&=&\{x\in V(G): x\text{ is a }k\text{-vertex not in }W_k\} \qquad (k\in \{3,4,5,6\}).
\end{eqnarray*}
For brevity, we use $W_{ij}$, $W_{ijk}$, and $W_{2345}$ to denote $W_{i}\cup  W_{j}$, $W_{i}\cup  W_{j}\cup W_k$, and $W_2\cup W_3\cup W_4\cup W_5$, respectively.
\begin{figure}[ht]
	\centering
  \includegraphics[scale=0.75,page=2]{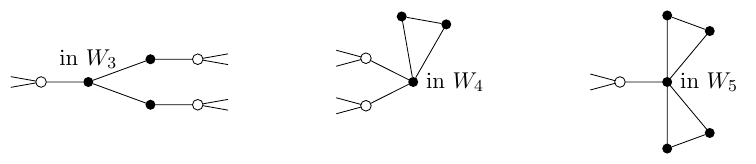}
  \caption{An illustration of a vertex in $W_3, W_4$, and $W_5$. }
  \label{fig:Wsets}
\end{figure}

The following is a list of  reducible configurations.
We will use {\bf[C1]}-{\bf[C10]} to show that every vertex has final charge exactly $\frac{8}{3}$ after applying our discharging rules.
The configurations {\bf[C${}^{\prime}$1]}-{\bf[C${}^{\prime}$5]} are utilized in the final step to reach a contradiction. We postpone the proofs to Sections~\ref{sec:proof:configuration} and~\ref{sec:proof:configuration2}.

\begin{enumerate}[{\bf [C1]}]
\item\label{rc-1vx} (Lemma~\ref{lem:rc-1}) A $1^{-}$-vertex.\vspace{-0.15cm}
\item\label{rc-22-c} (Lemma \ref{lem:rc-22-c}) Two adjacent 2-vertices not on a pendent triangle.\vspace{-0.15cm}
\item\label{rc-3-222} (Lemma~\ref{lem:rc-3}) A $3$-vertex with only 2-neighbors.\vspace{-0.15cm}
\item\label{rc-3p} (Lemma~\ref{lem:rc-3}) A $3$-vertex on a pendent triangle.\vspace{-0.15cm}
\item\label{rc-w3w3} (Lemma~\ref{lem:rc-w3w3})  Two adjacent $W_3$-vertices.\vspace{-0.15cm}
\item\label{rc-3-w3''}  (Lemma~\ref{lem:rc-w3''})  A 3-vertex with a 2-neighbor and a $W_3$-neighbor.\vspace{-0.15cm}
\item\label{rc-3-w3'} (Lemma~\ref{lem:rc-w3'}) A 3-vertex with two $W_3$-neighbors.\vspace{-0.15cm}
\item\label{rc-4p-w2w3w5} (Lemma~\ref{lem:rc-4p}) A $W_4$-vertex with a $W_{2345}$-neighbor.\vspace{-0.15cm}
\item\label{rc-w5} (Lemma~\ref{lem:rc-w5})
A $W_5$-vertex with either a $3$-neighbor or a $W_{25}$-neighbor. \vspace{-0.15cm}
\item\label{rc-w7} (Lemma~\ref{lem:rc-w7})
A 7-vertex on three pendent triangles with a $W_{235}$-neighbor.\vspace{-0.15cm}
\end{enumerate}
\begin{enumerate}[{\bf[C${}^{\prime}$1]}]
\item \label{rcp-cycle-w3-w2} (Lemma~\ref{lem:rc-cycle-w3-w2})
A cycle consisting of $W_{23}$-vertices.\vspace{-0.15cm}
\item\label{rcp-cycle-3-d4} (Lemma~\ref{lem:rc-cycle-3-d4}) A cycle consisting of $(V_3\cup W_4)$-vertices where every $V_3$-vertex has a $W_{23}$-neighbor.\vspace{-0.15cm}
\item \label{rcp-4-only2} (Lemma~\ref{lemma:rc-4-only2}) A $V_4$-vertex with all $W_{235}$-neighbors that has either two $W_2$-neighbors or two  $W_5$-neighbors. \vspace{-0.15cm}
\item \label{rcp-5-one-pendent}
(Lemma~\ref{lem:rc-5-one-pendent})
A $5$-vertex on one pendent triangle with three $W_{235}$-neighbors where two are $W_2$-neighbors. \vspace{-0.15cm}
\item \label{rcp-6-two-pendent}
(Lemma~\ref{lem:rc-6-two-pendent})
A 6-vertex on two pendent triangles with a $W_{235}$-neighbor and a different $W_{25}$-neighbor. \vspace{-0.15cm}
\end{enumerate}

We will use the discharging method.
For each vertex $v$ of $G$, let the {\it initial charge} $\mu(v)$ of $v$ be its degree, namely, $\mu(v)=\deg_G(v)$.
Note that the average initial charge (over all vertices) is at most $\frac{8}{3}$ since $\mad(G)\leq \frac{8}{3}$.
Next, we distribute the charge according to the following {\it discharging rules}, which are designed so that the total charge is preserved, to obtain the {\it final charge} $\mu^*(v)$ at each vertex $v$.
See Figure~\ref{fig:rules}.

\bigskip

\noindent{\bf Discharging Rules}
\begin{itemize}
\item[{\bf R1}] A $3^+$-vertex sends charge $\frac{2}{3}$ to each of its 2-neighbors on a pendent cycle.\vspace{-0.2cm}
\item[{\bf R2}] A $3^+$-vertex sends charge $\frac{1}{3}$ to each of its $W_2$-neighbors.\vspace{-0.2cm}
\item[{\bf R3}] A $3^+$-vertex sends charge $\frac{1}{3}$ to each of its $W_3$-neighbors.\vspace{-0.2cm}
\item[{\bf R4}] A $4^+$-vertex sends charge $\frac{1}{3}$ to each of its $W_5$-neighbors.
\end{itemize}

\begin{figure}[ht]
\centering
  \includegraphics[scale=0.75,page=3]{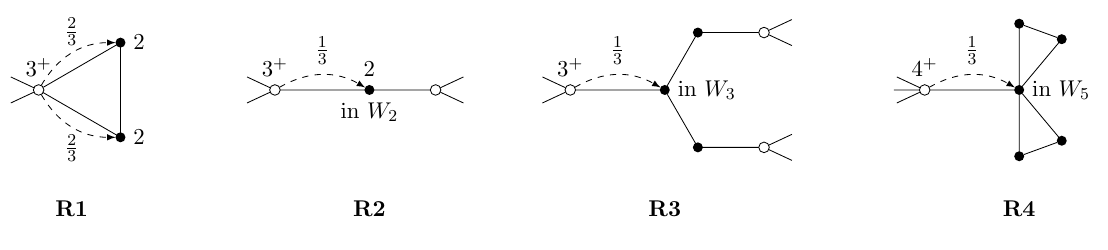}
  \caption{An illustration of the discharging rules. }
  \label{fig:rules}
\end{figure}

First, we will show that the final charge at each vertex is at least $\frac{8}{3}$.
Let $u$ be a $k$-vertex of $G$, and label the neighbors $u_1, \ldots, u_k$ of $u$ in such a way that $\deg_G(u_1)\leq \cdots \leq \deg_G(u_k)$.
By \rc{rc-1vx}, $k\geq 2$.
Note that $W_3$ is the set of all $3$-vertices with exactly two 2-neighbors by \rc{rc-3-222}, and moreover $W_3$ is an independent set by \rc{rc-w3w3}.
Also, by \rc{rc-22-c}, all pendent cycles of $G$ are triangles, and a 2-vertex with a 2-neighbor is on a pendent triangle.

\begin{itemize}
  \item[(1)] Assume $\deg_G(u)=2$.
  Note that a $2$-vertex does not send any charge by the discharging rules.
  If $u$ is on a pendent triangle, then it has a $3^+$-neighbor, which sends charge $\frac{2}{3}$ to $u$ by {\bf R1}.
  Thus, the final charge $\fc u$ of $u$ is at least $2+\frac{2}{3}=\frac{8}{3}$.
      If $u$ is not on a pendent triangle, then both neighbors of $u$ are $3^+$-vertices by \rc{rc-22-c}, so $\fc u= 2+2\cdot \frac{1}{3}=\frac{8}{3}$ by {\bf R2}.

  \item[(2)]
  Assume $\deg_G(u)=3$.
By \rc{rc-3-222}, $u$ has at most two $2$-neighbors, so the neighbor $u_3$ is a $3^+$-vertex.
 Also, $u$ is not on a pendent triangle by \rc{rc-3p}, so $u$ sends charge   $\frac{1}{3}$ to each $W_{23}$-neighbor.

\begin{itemize}
\item[(2-1)]
Suppose $u$ has exactly two 2-neighbors $u_1$ and $u_2$, that is, $u\in W_3$. So $u$ sends charge $1\over 3$ to each of $u_1$ and $u_2$ by {\bf R2}.
By \rc{rc-w3w3}, $u_3\not\in W_3$ and so $u$ does not send any charge to $u_3$. However, by {\bf R3} $u_3$ sends charge $\frac{1}{3}$ to $u$.
Hence $\fc u=3-2\cdot\frac{1}{3}+\frac{1}{3}=\frac{8}{3}$.

\item[(2-2)]
Suppose $u$ has exactly one 2-neighbor $u_1$. So $u$ sends charge $\frac{1}{3}$ to $u_1$ by {\bf R2}.
By \rc{rc-3-w3''}, $u_2,u_3\not\in W_3$, and so
 $u$ does not send any charge to $u_2,u_3$. Hence, $\fc u=3-\frac{1}{3}=\frac{8}{3}$.

\item[(2-3)] Suppose $u$ has no 2-neighbors. So $u$ sends charge only to $W_3$-neighbors.
By \rc{rc-3-w3'}, $u$ has at most one $W_3$-neighbor.
Thus, it sends charge at most $\frac{1}{3}$ by {\bf R3}, and so  $\fc u\ge 3- \frac{1}{3}=\frac{8}{3}$.
\end{itemize}

\item[(3)]
Assume $\deg_G(u)=4$.
If $u$ is on two pendent triangles, then the entire graph is formed by identifying two triangles at one vertex, which has an {\FII}.
If $u\in W_4$, then by \rc{rc-4p-w2w3w5}, $u$ sends charge only to neighbors on the pendent triangle.
Thus, $u$ sends charge $\frac{2}{3}$ to each of its 2-neighbors on the pendent triangle by {\bf R1}, so, $\fc u=4-2\cdot\frac{2}{3}=\frac{8}{3}$.
If $u$ is not on a pendent triangle, then it sends charge at most $\frac{1}{3}$ to each of its neighbors by {\bf R2}-{\bf R4}, so $\fc u\geq 4-4\cdot\frac{1}{3}=\frac{8}{3}$.

\item[(4)] Assume $\deg_G(u)=5$.
If $u \in W_5$, then by \rc{rc-w5}, the neighbor $u_5$ is a $4^+$-vertex, which is not in $W_5$.
By {\bf R4}, $u_5$ sends charge $\frac{1}{3}$  to $u$. Since $u_5\not\in W_5$, $u$ sends no charge to $u_5$.
By {\bf R1}, $u$ sends charge $\frac{2}{3}$ to each of its 2-neighbors on a pendent triangle, hence, $\fc u=5-4\cdot\frac{2}{3}+\frac{1}{3}=\frac{8}{3}$.
If $u$ is on exactly one pendent triangle, then $u$ sends charge $\frac{2}{3}$ to each of its $2$-neighbors on a pendent triangle by {\bf R1} and sends charge at most $\frac{1}{3}$ to each of the other neighbors by {\bf R2}-{\bf R4}, so $\fc u\ge 5-2\cdot\frac{2}{3}-3\cdot\frac{1}{3}=\frac{8}{3}$.
If $u$ not on a pendent triangle, then $\fc u\ge 5-5\cdot\frac{1}{3}=\frac{10}{3}>\frac{8}{3}$.

\item[(5)] Assume $\deg_G(u)\ge 6$.
Suppose $u$ is on exactly $k$ pendent triangles.
If $\deg_G(u)=2k$, then the entire graph is formed by identifying $k$ triangles at one vertex, which has an {\FII}.
Thus, $\deg_G(u)\ge 2k+1$, and so $\fc u\geq \deg_G(u)-(2k)\cdot\frac{2}{3}-(\deg_G(u)-2k)\cdot\frac{1}{3}=\frac{2\deg_G(u)-2k}{3}\geq \frac{\deg_G(u)+1}{3}$. So $\fc u\geq 3>\frac{8}{3}$ when $\deg_G(u)\geq 8$.
\begin{itemize}
\item[(5-1)]Suppose $\deg_G(u)= 7$.
By \rc{rc-w7}, $u$ is on at most two pendent triangles or $u$ has a neighbor who does not receive charge from $u$, so
$\fc u>\frac{8}{3}$.
\item[(5-2)] Suppose $\deg_G(u)=6$.
If $u$ is on at most one pendent triangles, so $\fc u\geq 6-2\cdot\frac{2}{3}-4\cdot\frac{1}{3}=\frac{10}{3}>\frac{8}{3}$.
If $u$ is on two pendent triangles, then $\fc u\geq 6-4\cdot\frac{2}{3}-2\cdot \frac{1}{3}=\frac{8}{3}$.
\end{itemize}
\end{itemize}

From (1)-(5), we conclude that the final charge of every vertex is at least $\frac{8}{3}$.
If there is a vertex whose final charge is more than $\frac{8}{3}$, then the average charge is more than $\frac{8}{3}$, which is a contradiction.
Hence, every vertex has final charge exactly $\frac{8}{3}$, which further implies the following {\bf[P1]}-{\bf[P5]}.

\begin{enumerate}[{\bf[P1]}]
\item\label{ra-6p}
By (5), there is no $7^+$-vertex.
Moreover, every $6$-vertex $v$ is on exactly two pendent triangles  and has two $W_{235}$-neighbors.
Together with \rcp{rcp-6-two-pendent}, $v$ has two $W_3$-neighbors.

\item\label{ra-5p} By (4), every $V_5$-vertex $v$ is on exactly one pendent triangle and has three $W_{235}$-neighbors.
Moreover, by \rcp{rcp-5-one-pendent}, $v$ has at most one $W_2$-neighbor.

\item\label{ra-4p} By (3), every $V_4$-vertex $v$ has only $W_{235}$-neighbors.
Moreover, by 
\rcp{rcp-4-only2},
$v$ has at most one $W_5$-neighbor and at most one $W_2$-neighbor.

\item\label{ra-4p-2}
By (4), \rc{rc-4p-w2w3w5}, \ra{ra-6p}, \ra{ra-5p}, and \ra{ra-4p}, every $W_4$-vertex $v$ has two $V_3$-neighbors.

\item\label{ra-3p} By (2),
every $V_3$-vertex $v$ has exactly one $W_{23}$-neighbor.
Moreover, by \rc{rc-w5}, \ra{ra-6p}, \ra{ra-5p}, and \ra{ra-4p},, the other two neighbors of $v$ are in $V_3\cup W_4$.
\end{enumerate}
 If $V_3\cup W_4\neq \emptyset$, then \ra{ra-4p-2} and \ra{ra-3p} imply that $G[V_3\cup W_4]$ is 2-regular.
 Yet, this contradicts \rcp{rcp-cycle-3-d4}, so $V_3\cup W_4=\emptyset$.
 Hence, $V(G)=T\sqcup W_{235} \sqcup \mathcal{V}_{4^+}$, where $\mathcal{V}_{4^+}=V_4\cup V_5 \cup V_6$  and $T$ denotes the set of all 2-vertices on pendent triangles of $G$.
 Note that by \ra{ra-6p}, \ra{ra-5p}, and \ra{ra-4p}, $\mathcal{V}_{4^+}$ is an independent set.

By \rc{rc-3-222}, \rc{rc-w3w3}, and \rcp{rcp-cycle-w3-w2},  $G[W_{23}]$ is the union of vertex-disjoint paths.
Let $Z$ be the set of isolated vertices of $G[W_{23}]$ and let $F_0$ be the set of non-isolated vertices of $G[W_{23}]$.
Note that by the definition of $W_3$, $Z\subset W_2$.

In the following, we will reach a  contradiction by finding an {\FII} of $G$.
We partition each of $\mathcal{V}_{4^+}$, $W_5$, and $T$ as in the following (1)-(3).
See Figure~\ref{fig:last:construction} for an illustration.
\begin{figure}[ht]
	\centering
  \includegraphics[scale=0.9,page=4]{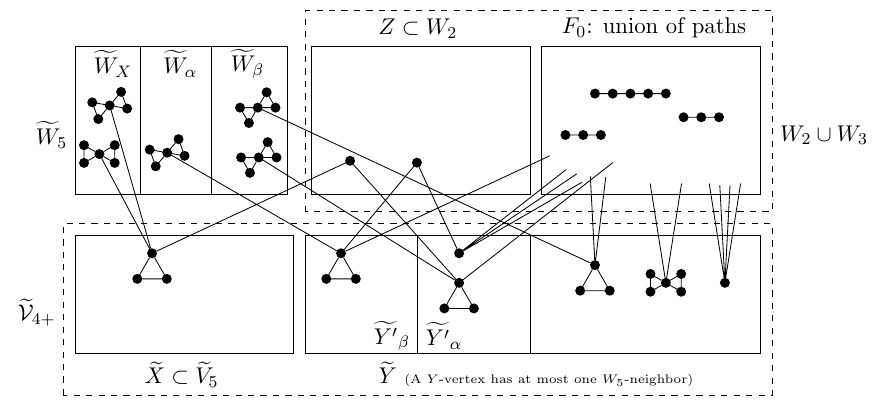}
  \caption{The structure of $G$. For $A\subset V(G)$,  $\widetilde{A}$ means $(N_G(A)\cap T) \cup A.$}
  \label{fig:last:construction}

\end{figure}

\noindent(1)
\underline{Partition $\mathcal{V}_{4^+}$ into $X\sqcup Y_{\alpha}\sqcup Y_{\beta}$.}
Let \begin{eqnarray*}
X&=&\{x\in \mathcal{V}_{4^+}  \mid  x\text{ has two }W_5\text{-neighbors} \},\\
Y&=&\{y\in \mathcal{V}_{4^+}  \mid  y\text{ has at most one }W_5\text{-neighbor} \}.
\end{eqnarray*}
Note that $X\subset V_5$  by \ra{ra-6p}, \ra{ra-5p}, and \ra{ra-4p}.
Let $Y'=\{y\in  Y  \mid  y\text{ has a }Z\text{-neighbor}\}$.
\ra{ra-6p}, \ra{ra-5p}, and \ra{ra-4p} also imply that every $Y'$-vertex has exactly one $Z$-neighbor, and  every $(Y\setminus Y')$-vertex  has only $(F_0\cup W_5)$-neighbors. 
Since $\mathcal{V}_{4^+}$ is an independent set, a $Y'$-vertex has degree one in $G[Y'\cup Z]$.
Since each $Z$-vertex is a 2-vertex, each component of $G[Y'\cup Z]$ is a path of length at most two.
Moreover, each $Y'$-vertex is an endpoint of some nontrivial component of
$G[Y'\cup Z]$, which implies that
each component of $G[Y'\cup Z]$ contains at most two $Y'$-vertices.
Partition $Y'$ into two sets $Y'_{\alpha}$ and $Y'_{\beta}$ so that for every component $C$ of $G[Y'\cup Z]$, $|C\cap Y'_\gamma|\le 1$ for every $\gamma \in\{\alpha,\beta\}$.
Now let $Y_{\alpha}=Y'_\alpha \cup(Y\setminus Y')$ and  $Y_{\beta}=Y'_\beta$ so that $X\sqcup Y_{\alpha}\sqcup Y_{\beta}$ is a partition of $\mathcal{V}_{4^+}$.

\noindent(2)
\underline{Partition $W_5$ into $W_X\sqcup W_{\alpha}\sqcup W_{\beta}$.} The following three sets partition $W_5$, since each $W_5$-vertex has exactly one neighbor in $X\cup Y$:
\begin{eqnarray*}
W_X&=&\{w\in W_5\mid w \text{ has an }X\text{-neighbor}\},\\
W_{\alpha}&=&\{ w\in W_5\mid w  \text{ has a }Y_\beta\text{-neighbor}\},\\
W_{\beta}&=&\{ w\in W_5\mid w  \text{ has a }Y_\alpha\text{-neighbor}\}.
\end{eqnarray*}

\noindent(3)
\underline{Partition $T$ into $T_X\sqcup T_{\alpha}\sqcup T_{\beta}$.}
Choose a 2-vertex from each pendent triangle at a $(W_X\cup X)$-vertex, and partition the chosen 2-vertices into two sets $T_{\alpha}$ and $T_\beta$ so that each of $T_\alpha$ and $T_\beta$ is a 2-independent set.
This is possible since a 2-vertex on a pendent triangle at an $X$-vertex and a 2-vertex on a pendent triangle at a $W_X$-vertex have  distance at least three.
Let $T_X=T\setminus(T_\alpha\cup T_\beta)$.

\medskip
We will now show that \fii{F}{I_\alpha}{I_\beta} is an {\FII} of $G$, where
\[F= W_{23} \cup X \cup W_X \cup T_X, \qquad
I_{\alpha}=Y_\alpha\cup W_\alpha\cup T_\alpha, \qquad
I_\beta=Y_\beta\cup W_\beta\cup T_\beta.\]
Consider $F$.
Since $W_{23}$ is the disjoint union of paths and $X$ has at most one $W_{23}$-neighbor, $W_{23}\cup X$ induces a forest.
Moreover, each pendent triangle containing a $T_X$-vertex also contains a vertex not in $F$.
Hence, $F$ induces a forest.

Suppose that there are two vertices $u,v\in I_\alpha\cup I_\beta$ with distance at most two.
By the definition of  $T_{\alpha}$ and $T_{\beta}$, if $u\in T_{\alpha}$ (resp. $T_\beta$), then $v\not\in I_{\alpha}$ (resp. $v\not\in I_\beta$).
Suppose that $u,v \in Y\cup W_{\alpha}\cup W_{\beta}$.
Since $u$ and $v$ have distance at most two,
at least one of $u$ and $v$ are in $Y$.
Suppose that $u,v\in Y$.
Since every vertex in $F_0\cup W_5$ has at most one $\mathcal{V}_{4^+}$-neighbor, it follows that $u,v\in Y'$ and they have a common $Z$-neighbor, which implies $u,v$ are in the same component of $G[Y\cup Z]$.
By way of construction, either $u\in Y'_{\alpha}$ and $v\in Y'_\beta$, or  $v\in Y'_{\alpha}$ and $u\in Y'_\beta$.
Lastly, suppose $u\in W_{\alpha}\cup W_{\beta}$ and $v\in Y$. Since $u$ and $v$ have distance at most two, $u$ and $v$ are adjacent. By the definition of $W_{\alpha}$ and $W_{\beta}$, $u$ and $v$ are not in the same $I_\gamma$ for some $\gamma\in \{\alpha, \beta\}$.

Hence, we have shown that \fii{F}{I_\alpha}{I_\beta} is an {\FII} of $G$, which is the final contradiction.

\section{Reducible Configurations [C1]-[C10]}\label{sec:proof:configuration}

In this section, we prove that  \rc{rc-1vx}-\rc{rc-w7} cannot exist in $G$.
Recall that the induction is on
(1) $|V^*(G)|$, the number of vertices of $G$ except the $2$-vertices on pendent cycles, and (2) $|V(G)|$, the number of vertices of $G$.

In all lemmas and claims, we often end up with an {\FII} of $G$, which is a contradiction.

\begin{lemma}\label{lem:rc-1}
In $G$, there is no $1^{-}$-vertex. \rc{rc-1vx}
\end{lemma}
\begin{proof}
If $G$ has a $1^{-}$-vertex $x$, then by the minimality of $G$, $G-x$ has an {\FII} \fii{F}{I_{\alpha}}{I_{\beta}}, which implies that
\fii{(F+x)}{I_{\alpha}}{I_{\beta}} is an {\FII} of $G$.
\end{proof}

\begin{lemma}\label{lem:rc-3}
In $G$, there is no $3$-vertex that either has only $2$-neighbors or is on a pendent triangle.
 \rc{rc-3-222}, \rc{rc-3p}
\end{lemma}

\begin{proof}
Let $v_1$, $v_2$, $v_3$ be the neighbors of a 3-vertex $v$.
Suppose to the contrary that $v$ has only $2$-neighbors.
Let for each $i\in [3]$, let $z_i$ be the neighbor of $v_i$ other than $v$.
Note that $z_i=z_j$ for some $i\neq j$ is possible, but it does not affect the following argument.
Let $S=\{v, v_1, v_2, v_3\}$ and $H=G-S$.
By the minimality of $G$, $H$ has an {\FII}  \fii{F}{I_{\alpha}}{I_{\beta}}.
If neither   \fii{(F+S-v)}{(I_\alpha+v)}{I_\beta} nor \fii{(F+S-v)}{I_\alpha}{(I_\beta+v)} is an {\FII} of $G$,
then $z_i\in I_\alpha$ and $z_j\in I_\beta$ for some $i,j\in[3]$.
Now, \fii{(F+S)}{ I_\alpha }{I_\beta} is an {\FII} of $G$.

Suppose to the contrary that $v$ is on a pendent triangle $vv_1v_2$.
Let $S=\{v_1,v_2\}$  and $H=G-S$.
By the minimality of $G$, $H$ has an {\FII} \fii{F}{I_{\alpha}}{I_{\beta}}.
If $v\in I_{\alpha}\cup I_{\beta}$, then \fii{(F+v_1v_2)}{ I_{\alpha} }{ I_{\beta}} is an {\FII} of $G$.
So assume  $v\in F$, and now either \fii{(F+v_1) }{ (I_{\alpha}+v_2 )}{ I_{\beta}} or
\fii{(F+v_1) }{ I_{\alpha} }{(I_{\beta}+v_2)} is an {\FII} of $G$.
\end{proof}

\begin{lemma}\label{lem:rc-cycle-3-d4-same1}
In $G$, the  following statements hold:
\begin{itemize}
\item[\rm (i)] There is no triangle $x_1x_2x_3$ such that $x_3\in W_2$, $x_1,x_2$ are $3$-vertices, and for each $i\in [2]$, the neighbor of $x_i$ other than $x_1$, $x_2$ is  either a $3^-$-vertex or a $W_4$-vertex.
\item[\rm (ii)] There is no $4$-cycle $x_1x_2x_3x_4$ such that $x_2\in W_2$, $x_1,x_3\in W_3$, and $x_4$ is a $3^-$-vertex.
\end{itemize}
\end{lemma}

\begin{proof} Suppose to the contrary that such a cycle exists.
 We use the labels as in  Figure~\ref{fig:sc2:lem1}.
\begin{figure}[ht]
	\centering
  \includegraphics[scale=0.75,page=5]{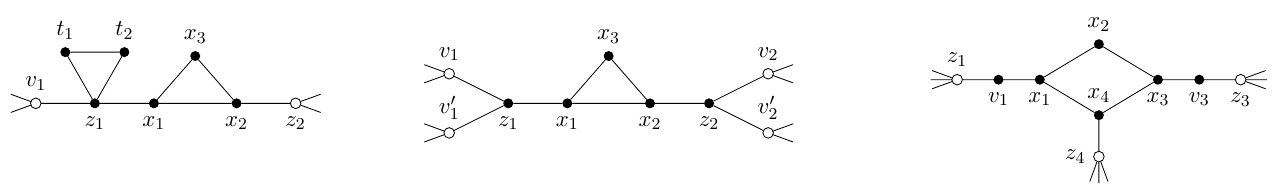} 
\caption{An illustration for Lemma~\ref{lem:rc-cycle-3-d4-same1}}
  \label{fig:sc2:lem1}

\end{figure}

\smallskip

\noindent (i) Let $S=\{x_1,x_2,x_3\}$ and $H=G-S$.
Since $\rho^*_H(z_1)+\rho^*_H(z_2)\ge \rho^*_H(z_1z_2)\ge =-4\cdot 3+3\cdot 5=3$ by~\eqref{eqref:basic:rho} and \eqref{eqref:basic:rho2}, without loss of generality assume $\rho^*_H(z_1)\ge 2$.
Then $\mad(H')\le \frac{8}{3}$, where  $H'$ is the graph obtained from $H$ by attaching $J_2$ to $z_1$. Since $|V^*(H')|<|V^*(G)|$,
by the minimality of $G$, $H'$ has an {\FII} \fii{F'}{I_{\alpha}'}{I_\beta'}.
Let $F=F'\cap V(H)$ $I_\alpha=I_\alpha'\cap V(H)$, and $I_\beta=I_\beta'\cap V(H)$.
By Lemma~\ref{lema:2B2} (ii),  $z_1\in F$.
If $z_2\not\in F$, then either \fii{(F+x_1x_2)}{I_\alpha}{(I_\beta+x_3)} or
\fii{(F+x_1x_2)}{(I_\alpha+x_3)}{I_\beta} is an {\FII} of $G$.
Thus, $\{z_1,z_2\}\subset F$.
If $z_1\in W_4$, then
we may assume that $t_2\in I_\alpha$, and then either \fii{(F+x_2x_3)}{(I_\alpha+x_1)}{I_\beta} or
\fii{(F+x_1x_2)}{(I_\alpha+x_3)}{I_\beta}
is an {\FII} of $G$.
If $z_1$ is a $3$-vertex, then either \fii{(F+x_1x_2)}{(I_\alpha+x_3)}{I_\beta},
\fii{(F+x_2x_3)}{I_\alpha}{(I_\beta+x_1)}, or  \fii{(F+x_2x_3)}{(I_\alpha+x_1)}{I_\beta} is an {\FII} of $G$.

\smallskip

\noindent (ii) Let $S=\{x_1,x_2,x_3,x_4,v_1,v_3\}$.
By the minimality of $G$, there is an {\FII} \fii{F}{I_\alpha}{I_\beta} of $G-S$.
If $z_1\in F$, then
\fii{(F+S-x_1x_3)}{(I_\alpha+x_1)}{(I_\beta+x_3)},
\fii{(F+S-x_1)}{(I_\alpha+x_1)}{I_\beta},
or \fii{(F+S-x_1)}{I_\alpha}{(I_\beta+x_1)} is an {\FII} of $G$.
Thus, $z_1,z_3\not\in F$, and now
\fii{(F+S-x_2)}{(I_\alpha+x_2)}{I_\beta} is an {\FII} of $G$.
\end{proof}

\begin{lemma}\label{lem:rc-w3w3}
In $G$, there are no two adjacent $W_3$-vertices. \rc{rc-w3w3}
\end{lemma}

\begin{proof}
Suppose to the contrary that  $x, y\in W_3$ are  adjacent. We use the labels as in the left figure of Figure~\ref{fig:22-c}.
Note that $x_1,x_2,y_1,y_2$ are distinct by Lemma \ref{lem:rc-cycle-3-d4-same1} (i).
It might happen that $z_i=z_j$ for some $i\neq j$, nonetheless the following arguments are still valid.
Let $H=G-S$ where $S=\{x,y,x_1,x_2,y_3,y_4\}$.
By the minimality of $G$,  $H$ has an {\FII} \fii{F}{I_{\alpha}}{I_{\beta}}.
If $z_1, z_2, z_3, z_4\not\in F$, then \fii{(F'+S)}{I_\alpha}{I_\beta} is an {\FII} of $G$.
Suppose that at least one  $z_i$ is in $F$.
Without loss of generality, we may assume $z_1, z_2\not\in I_\alpha$.
Since \fii{(F+S-x)}{(I_\alpha+x)}{I_\beta} is not an {\FII} of $G$,
$z_3, z_4\in F$. Then \fii{(F+S-xy)}{(I_\alpha+x)}{(I_\beta+y)} is an {\FII} of $G$.
\end{proof}

\begin{figure}[ht]
	\centering
  \includegraphics[scale=0.75,page=6]{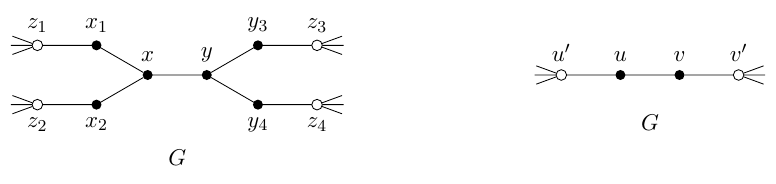}
    \caption{An illustration for Lemmas~\ref{lem:rc-w3w3} and~\ref{lem:rc-22-c}.}
  \label{fig:22-c}
\end{figure}

\begin{lemma}\label{lem:rc-22-c}
In $G$, there are no two adjacent $2$-vertices not on a pendent triangle.
\rc{rc-22-c}
\end{lemma}

\begin{proof}
Let $u$ and $v$ be adjacent 2-vertices not on a pendent triangle, and use the labels as in the right figure of Figure~\ref{fig:22-c}.
Let $H=G-\{u,v\}$.
Since $\rho^*_H(u')+ \rho^*_H(v') \ge 1$
by~\eqref{eqref:basic:rho} and \eqref{eqref:basic:rho2}, without loss of generality, we may assume $\rho^*_H(u')\ge 1$.
Now consider the graph $H'$ obtained by attaching a pendent triangle $u'xy$ to ${u'}$. Then $\mad(H')\leq {8\over 3}$.
Since
$|V^*(H')|<|V^*(G)|$, by the minimality of $G$, ${H'}$ has an {\FII} \fii{F'}{I'_\alpha}{I'_\beta}.
 Let $F=F'\cap V(H)$, $I_\alpha=I'_\alpha\cap V(H)$, and $I_\beta=I'_\beta\cap V(H)$.
If either $u'\not\in F$ or $v'\not\in F$, then \fii{(F+S)}{I_{\alpha}}{I_{\beta}} is an {\FII} of $G$.
If $u',v'\in F$, then without loss of generality we may assume $x\in I_{\alpha}$.
Now, \fii{(F+v)}{(I_{\alpha}+u)}{I_{\beta}}  is an {\FII} of $G$.
\end{proof}

\begin{lemma}\label{lem:rc-w5}
In $G$, there is no $W_5$-vertex with either a $3$-neighbor or a $W_{25}$-neighbor. \rc{rc-w5}
\end{lemma}

\begin{proof}
For $v\in W_5$, let $vt_1t_2$ and $vt_3t_4$ be the two pendent triangles at $v$, and let $v_1$ be the neighbor of $v$ that is not on a pendent triangle.
Suppose to the contrary that $v_1$ is either a $3$-vertex or a $W_{25}$-vertex.
If $v_1\in W_5$, then the entire graph $G$ is a subgraph of the graph $J_2$ in Figure~\ref{fig:2B_2}.
Yet, $J_2$ has an {\FII}, and therefore $G$ has an {\FII}.

Assume $v_1$ is a $3^-$-vertex.
Let $H=G-\{t_1, t_2, t_3, t_4, v\}$.
By the minimality of $G$, $H$ has an {\FII} \fii{F}{I_{\alpha}}{I_{\beta}}.
If $v_1\in F$, then \fii{(F+v t_1 t_3)}{(I_\alpha+t_2)}{(I_\beta+t_4)} is an {\FII} of $G$.
If $v_1\not\in F$ and we cannot move $v_1$ to $F$, then the neighbors of $v_1$ in $H$ are in $F$. Without of generality assume $v_1\in I_\alpha$.
Now \fii{(F+ t_1t_2t_3t_4)}{I_\alpha}{(I_\beta+v)} is an {\FII} of $G$.
\end{proof}

\begin{lemma}\label{lem:rc-4p}
In $G$, there is no vertex $v\in W_4$ with a $W_{2345}$-neighbor. \rc{rc-4p-w2w3w5}
\end{lemma}
\begin{proof}
Suppose to the contrary there is a vertex $v\in W_4$ with a $W_{2345}$-neighbor. We use the labels as in Figure~\ref{fig-rc-4p}.
\begin{figure}[ht]
	\centering
  \includegraphics[scale=0.75,page=7]{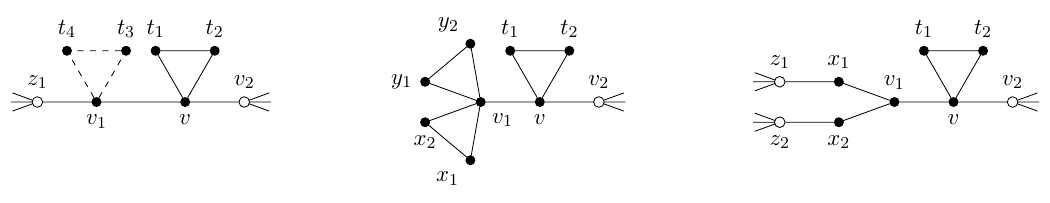}
    \caption{An illustration for Lemma~\ref{lem:rc-4p}.  }
    \label{fig-rc-4p}

\end{figure}

Assume $v_1\in W_{25}$.
Let $H=G-\{t_1,t_2\}$. By the minimality of $G$,  $H$ has an {\FII} \fii{F}{I_{\alpha}}{I_{\beta}}.
Since \fii{(F+t_1t_2)}{I_\alpha}{I_\beta} is not an {\FII} of $G$,  we have $v\in F$.
Also, since neither \fii{(F+t_2)}{(I_\alpha+t_1)}{I_\beta} nor \fii{(F+t_2)}{I_\alpha}{(I_\beta+t_1)} is an {\FII} of $G$, without loss of generality, we may assume $v_1\in I_\alpha$ and $v_2\in I_\beta$.
If $v_1\in W_2$, then \fii{(F+v_1t_2)}{(I_\alpha-v_1+t_1)}{I_\beta} is an {\FII} of $G$.
If $v_1\in W_4$, then we may assume $t_3,t_4\in F$, and so  \fii{(F+v_1t_2-t_4)}{(I_\alpha-v_1+t_1t_4)}{I_\beta} is an {\FII} of $G$.
If $v_1\in W_5$, then we may assume $\{x_1,x_2,y_1,y_2\}\subset F$ and so  $G$ has an {\FII} \fii{(F+v_1t_2-x_2y_2)}{(I'_\alpha+t_1x_2-v_1)}{(I'_\beta+y_2)}.

Now suppose $v_1\in W_3$.
Let  $S=\{v,v_1,t_1,t_2,x_1,x_2\}$ and let $H=G-S$.
\begin{claim}\label{claim:pendent:4:F}
For every {\FII} \fii{F}{I_\alpha}{I_\beta} of $H$,  $z_1,z_2\in F$ and $v_2\not\in F$.
\end{claim}
\begin{proof}
If $v_2\in F$,
then \fii{(F+S-t_1)}{(I_\alpha+t_1)}{I_\beta},
\fii{(F+S-t_1v_1)}{(I_\alpha+v_1)}{(I_\beta+t_1)}, or \fii{(F+S-t_1v_1)}{(I_\alpha+t_1)}{(I_\beta+v_1)} is an {\FII} of $G$. Thus $v_2\not\in F$, and we may assume $v_2\in I_\beta$.
If $z_i\not\in F$ for some $i$, then
\fii{(F+S-t_1)}{(I_\alpha+t_1)}{I_\beta} is an {\FII} of $G$.
\end{proof}

Suppose that $\rho^*_{H}(v_2)\ge 2$. Let $H'$ be the graph obtained from $H$ by identifying $v_2$ and $v^\ast$ in the graph $J_2$ in Figure~\ref{fig:2B_2}.
Then $\mad(H')\leq {8\over 3}$.
Since $|V^*(H')|<|V^*(G)|$,
by the minimality of $G$,  $H'$ has an {\FII}, which also gives an {\FII} \fii{F}{I_\alpha}{I_\beta} of $H$.
Now, $v_2\not\in F$ by Claim~\ref{claim:pendent:4:F}, but this contradicts Lemma~\ref{lema:2B2} (ii).
Hence, $\rho^*_H(v_2)\le 1$.
By~\eqref{eqref:basic:rho} and \eqref{eqref:basic:rho2},
\begin{eqnarray*}
&&\rho^*_{H}(z_1)+\rho^*_{H}(z_2)+1\ge \rho^*_{H}(z_1)+\rho^*_{H}(z_2)+\rho^*_{H}(v_2)\ge \rho^*_{H}(z_1z_2v_2)\ge -4\cdot6+3\cdot 9=3.
\end{eqnarray*}
Thus, we may assume that
$\rho^*_{H}(z_1)\ge 1$.
Let $H'$ be the graph obtained from $H$ by attaching a pendent triangle to $z_1$.
Then $\mad(H')\leq {8\over 3}$.
Since $|V^*(H')|<|V^*(G)|$,
by the minimality of $G$,  $H'$ has an {\FII}, which also gives an {\FII} \fii{F}{I_\alpha}{I_\beta} of $H$.
Thus, by Claim~\ref{claim:pendent:4:F},
we may assume $v_2\in I_\beta$ and $z_1, z_2\in F$.
By considering the pendent triangle of $H'$ at $z_1$, we know either \fii{(F+S-t_1x_1)}{(I_\alpha+t_1x_1)}{I_\beta} or \fii{(F+S-t_1x_1)}{(I_\alpha+t_1)}{(I_\beta+x_1)} is an {\FII} of $G$.
\end{proof}

\begin{lemma}\label{lem:rc-w3''}
In $G$, there is no $3$-vertex with a $2$-neighbor and a $W_3$-neighbor.
\rc{rc-3-w3''}
\end{lemma}

\begin{proof}
Suppose to the contrary that $v$ is a 3-vertex with a 2-neighbor and a $W_3$-neighbor. We use the labels as in Figure~\ref{fig:rc-w3''}.
Note that
it is easy to check $x_1$, $x_2$ are distinct from $v_2$. 
Let $H=G-S$, where $S=\{v,v_1,v_2,x_1,x_2\}$.

\begin{figure}[ht]
	\centering
  \includegraphics[scale=0.75,page=8]{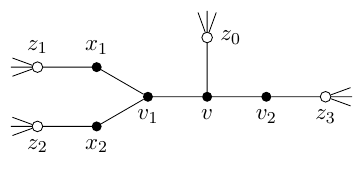} 
\caption{An illustration for Lemma~\ref{lem:rc-w3''}.}
    \label{fig:rc-w3''}

\end{figure}

\begin{claim}\label{claim:w3'':candidates}
For every {\FII} of \fii{F}{I_{\alpha}}{I_\beta} of $H$, we know $z_3\in F$.
Moreover,
if $z_0\in I_{\alpha}$ (resp. $I_\beta$), then one of $z_1$ and $z_2$ is in $F$ and the other is in $I_\beta$ (resp. $I_{\alpha}$).
\end{claim}
\begin{proof}
If $z_3\not\in F$, then \fii{(F+S)}{I_\alpha}{I_\beta},
\fii{(F+S-v_1)}{(I_\alpha+v_1)}{I_\beta}, or  \fii{(F+S-v_1)}{I_\alpha}{(I_\beta+v_1)} is an {\FII} of $G$.
Hence, $z_3\in F$.
Now, assume $z_0\in I_\alpha$.
Since neither \fii{(F+S-v_1)}{I_\alpha}{(I_\beta+v_1)} nor \fii{(F+S)}{I_\alpha}{I_\beta}
is an {\FII} of $G$, one of $z_1$ and $z_2$ is in $F$ and the other is in $I_\beta$.
\end{proof}
\begin{claim}\label{claim:pendent:w3'}
The following statements hold:
\begin{itemize}
\item[\rm(i)] Let $H'$ be the graph obtained from $H$ by attaching a pendent triangle $T$ to $z_0$.
If $H'$ has an {\FII} \fii{F'}{I'_\alpha}{I'_\beta}, then $z_0\not\in F'$.
\item[\rm(ii)]Let $H'$ be the graph obtained from $H$ by adding an edge $z_iz_j$ for some $i,j\in \{1,2,3\}$.
If $H'$ has an {\FII} \fii{F'}{I'_\alpha}{I'_\beta}, then $z_0\in F'$.
\end{itemize}
\end{claim}
\begin{proof}
For an {\FII} \fii{F'}{I'_\alpha}{I'_\beta} of $H'$, let $F=F'\cap V(H)$, $I_{\alpha}=I'_{\alpha}\cap V(H)$,
and $I_{\beta}=I'_{\beta}\cap V(H)$.
Since \fii{F}{I_\alpha}{I_\beta} is also an {\FII} of $H$, Claim~\ref{claim:w3'':candidates}  implies $z_3\in F$.

\smallskip

\noindent (i) Suppose to the contrary that $z_0\in F$.
Without loss of generality, we may assume a 2-vertex on $T$ belongs to
$I'_{\alpha}$.
Thus, either
\fii{(F'+S-v)}{(I'_{\alpha}+v)}{I_\beta'}
or \fii{(F+S-vv_1)}{(I_{\alpha}+v)}{(I_\beta+v_1)} is an {\FII} of $G$.

\smallskip

\noindent (ii) Suppose to the contrary that $z_0\not\in F$, say $z_0\in I_{\alpha}$.
By Claim~\ref{claim:w3'':candidates},
without loss of generality, assume  $z_1\in F$ and $z_2\in I_{\beta}$.
If $z_iz_j=z_1z_2$, then   \fii{(F+S-x_1)}{I_\alpha}{(I_\beta+x_1)} is an {\FII} of $G$.
If $z_iz_j=z_1z_3$, then \fii{(F+S)}{I_\alpha}{I_\beta} is an {\FII} of $G$.
If $z_iz_j=z_2z_3$, then  \fii{(F+S-v_2)}{I_\alpha}{(I_\beta+v_2)}   is an {\FII} of $G$.
\end{proof}

Let $Z=\{z_0,z_1,z_2,z_3\}$.
By  \eqref{eqref:basic:rho2},
$\rho^*_H(Z)\ge -4\cdot 5+3\cdot 8=4$.

\begin{claim}\label{claim:w3''-potential}
The following statements hold:
\begin{itemize}
\item[\rm(a)] $\rho^*_H(z_i)=1$ for $i\in\{0,3\}$.
\item[\rm(b)] $\rho^*_H(z_iz_j)\ge 2$ for $i,j\in\{0,1,2,3\}$ with $i\neq j$.
\end{itemize}
\end{claim}
\begin{proof}
Instead of proving (a) and (b) separately, we show the following (1)-(4):
\[\text{(1) } \rho^*_H(z_3)\leq 1.
\quad
\text{(2) } \rho^*_H(z_0)\leq 1,
    \quad
\text{(3) }  \rho^*_H(z_iz_1z_2)\le 3 \text{ for } i\in\{0,3\},\quad
\text{(4) } \rho^*_H(z_i z_0z_3)\le 3\text{ for } i\in\{1,2\}.\]
We argue it is sufficient to show (1)-(4).
From (3) and \eqref{eqref:basic:rho},   $\rho^*_H(z_3)+3\ge \rho^*_H(z_3)+\rho^*_H(z_0z_1z_2)\geq \rho^*_H(Z)\geq 4$. Thus $\rho^*_H(z_3)\geq 1$, which implies $\rho^*_H(z_3)=1$ by (1).
Similarly, (2) and (3) imply $\rho^*_H(z_0)=1$.
Hence, (1), (2), and (3) imply (a).
We now show how (3) and (4) imply (b).
Note that (3) and (4) are equivalent to  $\rho^*_H(z_iz_jz_k)\le 3$ for three distinct $i,j,k\in\{0,1,2,3\}$.
For $i,j\in\{0,1,2,3\}$,
by (3) and (4), since  \eqref{eqref:basic:rho} implies \[\rho^*_H(z_{i}z_{j})+6\geq\rho^*_H(z_{i}z_{j})+\rho^*_H(Z\setminus\{z_{i}\})+\rho^*_H(Z\setminus\{z_{j}\})\geq
\rho^*_H(z_{i}z_{j})+\rho^*_H(Z)+\rho^*_H(Z\setminus\{z_{i},z_{j}\})\geq 2\rho^*_H(Z)=8,\] we know $\rho^*_H(z_iz_j)\geq 2$.
Hence, it is sufficient to show (1)-(4).

In each case, we will define a graph $H'$ from $H$ so that $\mad(H')\le \frac{8}{3}$ and $|V^*(H')|<|V^*(G)|$.
By the minimality of $G$, $H'$ has an {\FII} \fii{F'}{I'_\alpha}{I'_\beta}.
Let $F=F'\cap V(H)$, $I_\alpha=I'_\alpha\cap V(H)$, and $I_\beta=I'_\beta\cap V(H)$.
Claim~\ref{claim:w3'':candidates} implies $z_3\in F$.

(1) Suppose to the contrary that $\rho^*_H(z_3)\ge 2$.
Let $H'$ be the graph obtained from $H$ by attaching two pendent triangles to $z_3$.
Since $\rho^*_H(z_3)\ge 2$, $\mad(G')\le\frac{8}{3}$.
Note that, since $z_3\in F$, $z_3$ has only $F$-neighbors in $H$.
Assume $z_0\in F$.
If $z_1,z_2\not\in I_\alpha$, then \fii{(F+S-v_1v_2)}{(I_{\alpha}+v_1)}{(I_\beta+v_2)} is an {\FII} of $G$.
Otherwise, either \fii{(F+S-v_2)}{I_{\alpha}}{(I_\beta+v_2)} or \fii{(F+S-v_1v_2)}{(I_{\alpha}+v_2)}{(I_\beta+v_1)} is an {\FII} of $G$.
Without loss of generality, assume $z_0\in I_\alpha$.
By Claim~\ref{claim:w3'':candidates}, \fii{(F+S-v_2)}{I_{\alpha}}{(I_\beta+v_2)} is an {\FII} of $G$.

(2) Suppose to the contrary that $\rho^*_H(z_0)\ge 2$.
Let $H'$ be the graph obtained from $H$ by attaching $J_1$ and a pendent triangle to $z_0$.
Note that $\mad(H')\le \frac{8}{3}$ since $\rho^*_H(z_0)\ge 2$.
By Claim~\ref{claim:pendent:w3'} (i),  $z_0\not\in F$.
Without loss of generality, assume $z_0\in I_{\alpha}$. By Claim~\ref{claim:w3'':candidates}, $z_1\in F$ and $z_2\in I_{\beta}$.
Also, $z_0$ has a $I'_{\beta}$-neighbor by Lemma~\ref{lema:2B2} (i).
Thus, \fii{(F+S-v)}{I_\alpha}{(I_\beta+v)} is an {\FII} of $G$.

(3) Suppose to the contrary that $\rho^*_H(z_{i}z_1z_2)\ge 4$ for some $i\in \{0,3\}$.
By (1), (2), and \eqref{eqref:basic:rho}, $1+\rho^*_H(z_{1}z_2)\geq \rho^*_H(z_i)+\rho^*_H(z_{1}z_2)\ge 4$, so  $\rho^*_H(z_{1}z_2)\ge 3$.
Therefore,
\begin{eqnarray}\label{eq:potential:i12}
&& \rho^*_H(z_i)=1,  \rho^*_H(z_1z_2)\ge 3, \rho^*_H(z_iz_1z_2)\ge 4.
\end{eqnarray}
Let $H'$ be the graph obtained from $H$ by attaching one pendent triangle $T$ to $z_i$ and adding an edge $z_1z_2$.
Note that by~\eqref{eq:potential:i12},   $\mad(H')\le \frac{8}{3}$.
By Claim~\ref{claim:pendent:w3'} (i) and (ii), it must be that $i=3$.

If $z_0\in F$, then we may assume a 2-vertex of $T$ belongs to $I'_{\alpha}$.
Furthermore, if \fii{(F+S-v_1v_2)}{(I_\alpha+v_2)}{(I_\beta+v_1)} is not an {\FII} of $G$, then either
\fii{(F+S-v_2x_1)}{(I_\alpha+v_2)}{(I_\beta+x_1)} or \fii{(F+S-v_2x_2)}{(I_\alpha+v_2)}{(I_\beta+x_2)} is an {\FII} of $G$.
Now, without loss of generality, assume $z_0\in I_\alpha$.
By Claim~\ref{claim:pendent:w3'} (ii), we may assume $z_1\in F$ and $z_2\in I_\beta$.
Now, \fii{(F+S-x_1)}{I_\alpha}{(I_\beta+x_1)} is an {\FII} of $G$.

(4) Without loss of generality, suppose to the contrary that $\rho^*_H(z_{0}z_1z_{3})\ge 4$.
Since \eqref{eqref:basic:rho} implies $\rho^*_H(z_0)+\rho^*_H(z_{1}z_3)\ge \rho^*_H(z_0z_1z_3) \ge 4$, together with (a)
(which is true since (1), (2), and (3) are proved), we know $\rho^*_H(z_{1}z_3)\ge 3$.
Therefore,
\begin{eqnarray}\label{eq:potential:013}
&& \rho^*_H(z_0)=1,  \rho^*_H(z_{1}z_3)\ge 3, \rho^*_H(z_{0}z_1z_3)\ge 4.
\end{eqnarray}
Note that by \eqref{eq:potential:013},  $\mad(H')\le \frac{8}{3}$, where  $H'$ is the graph obtained from $H$ by attaching one pendent triangle to $z_0$
and adding an edge $z_1z_3$.
By the minimality of $G$, $H'$ has an {\FII}, which is a contradiction by Claim~\ref{claim:pendent:w3'} (i) and (ii).
\end{proof}

By Claim~\ref{claim:w3''-potential}~(a) and \eqref{eqref:basic:rho}, we have $\rho^*_H(z_{0}z_1z_{2})\geq\rho^*_H(Z)-\rho^*_H(z_{3})\ge 3$.
In addition, for $i\in\{1,2\}$, since  $\rho^*_H(z_i)+1=\rho^*_H(z_i)+\rho^*_H({z_0})\ge \rho^*_H(z_{0}z_{i})\ge 2$, we have $\rho^*_H(z_{i})\ge 1$.
Therefore, by  Claim~\ref{claim:w3''-potential}
\[\rho^*_H(z_0)=1, \rho^*_H(z_1)\geq 1, \rho^*_H(z_2)\geq 1, \rho^*_H(z_{0}z_{1})\geq 2, \rho^*_H(z_{0}z_{2})\geq 2, \rho^*_H(z_{1}z_{2})\geq 2, \rho^*_H(z_{0}z_1z_{2})\geq 3.\]
Then  $\mad(H')\le \frac{8}{3}$, where
$H'$ is the graph obtained from $H$ by attaching a pendent triangle to each of $z_0$, $z_1$, $z_2$.
Since $|V^*(H')|<|V^*(G)|$, by the minimality of $G$, ${H'}$ has an  {\FII}, which also gives an {\FII}  \fii{F}{I_{
\alpha}}{I_\beta} of $H$.
By Claim~\ref{claim:pendent:w3'} (i), $z_0\not\in F$.
Together with Claim~\ref{claim:w3'':candidates}, we may assume $z_0\in I_{\alpha}$, $z_1\in I_\beta$, and $z_2\in F$.
By considering the pendent triangle at  $z_2$,   either \fii{(F+S-x_2)}{(I_{\alpha}+x_2)}{I_\beta} or \fii{(F+S-x_2)}{I_{\alpha}}{(I_\beta+x_2)} is an {\FII} of $G$.
\end{proof}

\begin{lemma}\label{lem:rc-w3'}
In $G$, there is no $3$-vertex with two $W_3$-neighbors.
\rc{rc-3-w3'}
\end{lemma}

\begin{proof}
Suppose to the contrary that there is a 3-vertex $v$ with two $W_3$-neighbors.
We use the labels as in Figure~\ref{fig:rc-w3'}.
By Lemma~\ref{lem:rc-cycle-3-d4-same1} (ii), all $x_i$'s are distinct.
Let $H=G-S$ where  $S=\{v,v_1,v_2,x_1,x_2,x_3,x_4\}$.

\begin{figure}[ht]
	\centering
  \includegraphics[page=9,scale=0.75]{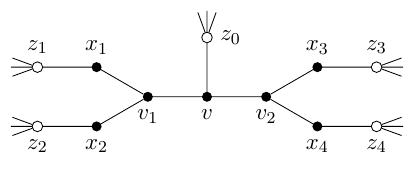}
\caption{An illustration for Lemma~\ref{lem:rc-w3'}.}
    \label{fig:rc-w3'}

\end{figure}

\begin{claim}\label{claim:w':candidates}
For every {\FII} \fii{F}{I_{\alpha}}{I_{\beta}} of $H$, the following statements hold: 
\begin{itemize}
\item[\rm(i)] if $z_0\in F$, then exactly one of $z_1,z_2$ is in $F$,
exactly one of $z_3,z_4$ is in $F$, and $\{z_1,z_2,z_3,z_4\}\cap I_\gamma=\emptyset$ for some  $\gamma\in \{\alpha,\beta\}$.
\item[\rm(ii)]  if $z_0\in I_{\gamma}$ for some $\gamma\in \{\alpha,\beta\}$  and $\{z_1,z_2,z_3,z_4\}\not\subset  F$, then exactly one of $z_1,z_2$ is in $F$,
exactly one of $z_3,z_4$ is in $F$, and $\{z_1,z_2,z_3,z_4\}\cap I_\gamma=\emptyset$.
\end{itemize}
\end{claim}
\begin{proof}
(i) Assume $z_0\in F$.
If $z_1,z_2\not\in F$, then
\fii{(F+S-v_2)}{(I_{\alpha}+v_2)}{I_{\beta}},
\fii{(F+S-v_2)}{I_{\alpha}}{(I_{\beta}+v_2)}, or
\fii{(F+S)}{I_{\alpha}}{I_{\beta}}
is an {\FII} of $G$.
If $z_1,z_2\in F$, then
 \fii{(F+S-v_1)}{(I_{\alpha}+v_1)}{I_{\beta}}, \fii{(F+S-v_1v_2)}{(I_{\alpha}+v_1)}{(I_{\beta}+v_2)}, or \fii{(F+S-v_1v_2)}{(I_{\alpha}+v_2)}{(I_{\beta}+v_1)} is an {\FII} of $G$.
Thus exactly one of $z_1$ and $z_2$ is in $F$.
By symmetry, exactly one of $z_3$ and $z_4$ is in $F$.
Now, if $\{z_1,z_2,z_3,z_4\}\cap I_\gamma \neq \emptyset$ for each $\gamma\in\{\alpha,\beta\}$, then either \fii{(F+S-v_1v_2)}{(I_{\alpha}+v_1)}{(I_{\beta}+v_2)} or \fii{(F+S-v_1v_2)}{(I_{\alpha}+v_2)}{(I_{\beta}+v_1)} is an {\FII} of $G$.

\smallskip

\noindent (ii) Without loss of generality, assume $z_0\in I_{\alpha}$.
If $z_1,z_2\in F$, then $|\{z_3,z_4\}\cap F|\le 1$, so \fii{(F+S-v_1)}{I_{\alpha}}{(I_\beta +v_1)} is an {\FII} of $G$.
If $z_1,z_2\not\in F$, then either \fii{(F+S)}{I_{\alpha}}{I_\beta} or \fii{(F+S-v_2)}{I_{\alpha}}{(I_\beta+v_2)} is an {\FII} of $G$.
Thus, exactly one of $z_1,z_2$ is in $F$.
By symmetry, exactly one of $z_3$ and $z_4$ is in $F$.
If $\{z_1,z_2,z_3,z_4\}\cap I_\alpha \neq \emptyset$, then either 
\fii{(F+S-v_1)}{I_{\alpha}}{(I_\beta+v_1)}
or \fii{(F+S-v_2)}{I_{\alpha}}{(I_\beta+v_2)} is an {\FII} of $G$.
\end{proof}

In the proof of each of the following cases, we will define a graph $H'$ by modifying $H$ so that $\mad(H')\le \frac{8}{3}$ and $|V^*(H')|<|V^*(G)|$.
By the minimality of $G$, $H'$ has an {\FII} \fii{F'}{I'_\alpha}{I'_\beta}.
Let $F=F'\cap V(H)$, $I_\alpha=I'_\alpha\cap V(H)$, and $I_\beta=I'_\beta\cap V(H)$.
Then we can apply Claim~\ref{claim:w':candidates}.
Note that by \eqref{eqref:basic:rho2},
we know $\rho^*_H(Z)\ge -4\cdot 7+3\cdot 11= 5$, where $Z=\{z_0,z_1,z_2,z_3,z_4\}$.

\begin{claim}\label{claim:w'-potential}
$\rho^*_H(z_1z_2)=\rho^*_H(z_{3}z_4)= 2$ and $\rho^*_H(z_0)= 1$.
\end{claim}
\begin{proof}
By~\eqref{eqref:basic:rho}, since $\rho^*_H(z_1z_2)+\rho^*_{H}(z_3z_4)+\rho^*_H(z_0)\ge \rho^*_H(Z)\ge 5$,
it is sufficient to show  $\rho^*_H(z_1z_2)\le 2$, $\rho^*_H(z_3z_4)\le 2$, and $\rho^*_H(z_{0})\le 1$.

Suppose to contrary that $\rho^*_H(z_1z_2)\ge 3$.
Let $H'$ be the graph obtained from $H$ by adding an edge $z_1z_2$.
Note that since  $\rho^*_H(z_1z_2)\ge 3$,  $\mad(H')\leq {8\over 3}$.
If $z_0\in F$, then by Claim~\ref{claim:w':candidates}~(i), we may assume that $z_1,z_3\in F$ and $z_2,z_4\in I_{\alpha}$, and therefore
\fii{(F+S-x_1v_2)}{(I_{\alpha}+x_1)}{(I_\beta+v_2)} is an {\FII} of $G$.
Suppose that $z_0\not\in F$, say
$z_0\in I_{\alpha}$.
If $\{z_1,z_2,z_3,z_3\}\not \subset F$, then by
Claim~\ref{claim:w':candidates}~(ii), we may assume that $z_1,z_3\in F$, $z_2,z_4\in I_\beta$, and therefore
\fii{(F+S-x_1)}{I_{\alpha}}{(I_\beta+x_1)} is an {\FII} of $G$.
If $\{z_1,z_2,z_3,z_3\}\subset F$, then  \fii{(F+S-v_2)}{I_{\alpha}}{(I_\beta+v_2)} is an {\FII} of $G$.
Therefore, $\rho^*_H(z_1z_2)\le 2$, and by symmetry, $\rho^*_H(z_{3}z_4)\le 2$.

In the following, we will show $\rho^*_H(z_0)\le 1$ in two steps.
First we show $\rho^*_H(z_0)\le 2$, and then show $\rho^*_H(z_0)\neq 2$.
Suppose to the contrary that $\rho^*_H(z_0)\ge 3$.
Let $H'$ be the graph obtained from $H$ by attaching $J_2$ and one pendent triangle to $z_0$. See the left figure of Figure~\ref{fig:rc-w3'-2}.
Note that $\mad(H')\le \frac{8}{3}$ since $\rho^*_H(z_0)\ge 3$.
By Lemma~\ref{lema:2B2} (i), we know $z_0\in F$. So Claim~\ref{claim:w':candidates} (i) applies, and moreover, by considering the pendent triangle at $z_0$, we know that either
\fii{(F+S-v)}{I_{\alpha}}{(I_\beta+v)}
or
\fii{(F+S-v)}{(I_{\alpha}+v)}{I_\beta} is an {\FII} of $G$.
Hence, $\rho^*_H(z_0)\le 2$.

Now suppose that $\rho^*_H(z_0)=2$.
By~\eqref{eqref:basic:rho},
\begin{eqnarray*}
\rho^*_H(z_{0}z_{1})+\rho^*_H(z_{0}z_2)+\rho^*_H(z_{0}z_3)+\rho^*_H(z_{0}z_4)&\ge&
\rho^*_H(z_{0}z_1z_2)+\rho^*_H(z_0)+\rho^*_H(z_{0}z_3z_4)+\rho^*_H(z_0)\\
&\ge& \rho^*_H(Z)+3\rho^*_H(z_0)\geq 5+3\cdot2=11,
\end{eqnarray*}
so without loss of generality, we may assume  $\rho^*_H(z_0z_1)\ge 3$.
Since  $\rho^*_H(z_0)+\rho^*_H(z_1)\ge \rho^*_H(z_{0}z_1)$ by \eqref{eqref:basic:rho} and we already have $\rho^*_H(z_0)\le 2$,
it follows that $\rho^*_H(z_1)\ge 1$.
Hence,
\begin{eqnarray}\label{eq:w3'}
&&\rho^*_H(z_0)=2, \rho^*_H(z_1)\ge 1, \rho^*_H(z_{0}{z_1})\ge 3.
\end{eqnarray}
Let $H'$ be the graph obtained from $H$ by attaching $J_1$ and one pendent triangle $T_0$  to $z_0$, and attaching one pendent triangle $T_1$ to $z_1$. See the middle figure of Figure~\ref{fig:rc-w3'-2}.
Note that by \eqref{eq:w3'},  $\mad(H')\le \frac{8}{3}$.
If $z_0\in F$, then by considering $T_0$, Claim~\ref{claim:w':candidates} (i) implies that
either \fii{(F+S-v)}{(I_{\alpha}+v)}{I_\beta}
or
\fii{(F+S-v)}{I_{\alpha}}{(I_\beta+v)}
 is an {\FII} of $G$.
Suppose that $z_0\not\in F$,
say $z_0\in I_{\alpha}$.
By Lemma~\ref{lema:2B2} (i), the neighbor of $z_0$ in $J_1$ is in $I'_\beta$.
If $\{z_1, z_2, z_3, z_4\}\not\subset F$,
then  by  Claim~\ref{claim:w':candidates} (ii), \fii{(F+S-v)}{I_{\alpha}}{(I_\beta+v)} is an {\FII} of $G$.
If $\{z_1, z_2, z_3, z_4\}\subset F$, then by considering $T_1$, either \fii{(F+S-x_1v_2)}{(I_{\alpha}+x_1)}{(I_\beta+v_2)} or \fii{(F+S-x_1v_2)}{I_{\alpha}}{(I_\beta+x_1v_2)} is an {\FII} of $G$.
Thus, $\rho^*_H(z_0)\neq 2$ and so $\rho^*_H(z_0)\leq 1$.
\end{proof}

\begin{figure}[ht]
	\centering
  \includegraphics[page=10,scale=0.9]{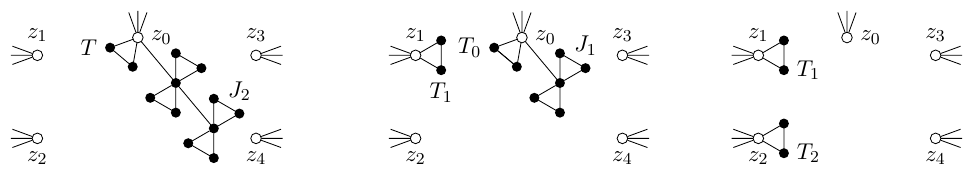}
\caption{An illustration for Lemma~\ref{lem:rc-w3'}.}
    \label{fig:rc-w3'-2}

\end{figure}

\begin{claim}\label{claim:w3'-pontential2}
Either $\rho^*_H(z_1)=0$ or $\rho^*_H(z_2)=0$, and also either  $\rho^*_H(z_3)=0$ or $\rho^*_H(z_4)=0$.
\end{claim}
\begin{proof}
Suppose to the contrary that $\rho^*_H(z_1)\ge 1$ and $\rho^*_H(z_2)\ge 1$. Note that $\rho^*_H(z_1z_2)=2$ by Claim~\ref{claim:w'-potential}.
Let $H'$ be the graph obtained from $H$ by attaching one pendent triangle to each of $z_1$ and $z_2$. See the right figure of Figure~\ref{fig:rc-w3'-2}.
Note that since $\rho^*_H(z_1),\rho^*_H(z_2)\ge 1$, and $\rho^*_H(z_1z_2)=2$,  $\mad(H')\leq {8\over 3}$.
If $z_0\in F$,  then by Claim~\ref{claim:w':candidates} (i) we may assume that $z_1,z_3\in F$ and $z_2,z_4\in I_{\alpha}$, and therefore
either \fii{(F+S-x_1v_2)}{(I_{\alpha}+x_1)}{(I_\beta+v_2)} or \fii{(F+S-x_1v_2)}{I_{\alpha}}{(I_\beta+x_1v_2)} is an
{\FII} of $G$.
Suppose that $z_0\not\in F$, say $z_0\in I_\alpha$.
If $\{z_1, z_2, z_3, z_3\}\not\subset F$, then by Claim~\ref{claim:w':candidates} (ii), we may assume $z_1\in F$, and therefore
either \fii{(F+S-x_1)}{(I_{\alpha}+x_1)}{I_\beta} or \fii{(F+S-x_1)}{I_{\alpha}}{(I_\beta+x_1)} is an {\FII} of $G$.
If $\{z_1, z_2, z_3, z_3\}\subset F$, then either \fii{(F+S-x_1v_2)}{(I_{\alpha}+x_1)}{(I_\beta+v_2)} or \fii{(F+S-x_1v_2)}{I_{\alpha}}{(I_\beta+x_1v_2)} is an {\FII} of $G$.
Hence, either $\rho^*_H(z_1)=0$ or $\rho^*_H(z_2)=0$, and by symmetry, either  $\rho^*_H(z_3)=0$ or $\rho^*_H(z_4)=0$.
\end{proof}

By Claim~\ref{claim:w3'-pontential2}, we may assume that $\rho^*_H(z_1)=\rho^*_H(z_3)=0$.
Let $H'$ be the graph obtained from $H$ by adding a path of length two between $z_2$ and $z_4$, and for each of $z_2$ and $z_4$, attach one pendent triangle $T_2$ and $T_4$, respectively.
Note that $\mad(H')\leq {8\over 3}$, since  adding a path of length two decreases the potential by 2, and the following inequalities, which follow from Claim~\ref{claim:w'-potential} and \eqref{eqref:basic:rho}:
\begin{eqnarray*}
&&\rho^*_H(z_2)= \rho^*_H(z_1)+\rho^*_H(z_2)\ge \rho^*_H(z_{1}z_{2})=2\\
&&\rho^*_H(z_4)= \rho^*_H(z_3)+\rho^*_H(z_4)\ge \rho^*_H(z_{3}z_{4})=2\\
&&\rho^*_H(z_{2}z_{4})+1=\rho^*_H(z_{2}z_{4})+\rho^*_H(z_1)+\rho^*_H(z_3)+\rho^*_H(z_0)\ge \rho^*_H(Z)\ge 5.
\end{eqnarray*}
If $z_2,z_4\not\in F$, then we may assume $z_2\in I_\alpha$ and $z_4\in I_\beta$ since $z_2$ and $z_4$ have distance two in $H'$.
Therefore, \fii{(F+S-v_1v_2)}{(I_\alpha+v_2)}{(I_\beta+v_1)},
\fii{(F+S-v_1)}{I_\alpha}{(I_\beta+v_1)},
or \fii{(F+S-v_2)}{(I_\alpha+v_2)}{I_\beta}  is an {\FII} of $G$.
Suppose that $z_2\in F$.
Without loss of generality, we may assume a 2-vertex on $T_2$ belongs to $I'_{\alpha}$, where \fii{F'}{I'_\alpha}{I'_\beta} is an {\FII} of $H'$.
If $z_0\in F$, then by Claim~\ref{claim:w':candidates} (i),   either \fii{(F+S-x_2v_2)}{(I_{\alpha}+x_2)}{(I_\beta+v_2)} or \fii{(F+S-x_2v_2)}{(I_{\alpha}+x_2v_2)}{I_\beta} is an {\FII} of $G$.
Suppose that $z_0\not\in F$.
If $\{z_1,z_2,z_3,z_4\}\not\subset F$, then by Claim~\ref{claim:w':candidates} (ii), \fii{(F+S-x_2)}{(I_{\alpha}+x_2)}{I_\beta}  is an {\FII} of $G$.
If $\{z_1,z_2,z_3,z_4\}\subset F$,
then \fii{(F+S-x_2v_2)}{(I_{\alpha}+x_2)}{(I_\beta+v_2)}
is an {\FII} of $G$.
\end{proof}

\begin{lemma}\label{lem:rc-w7}
In $G$, there is no $7$-vertex on three pendent triangles with  a $W_{235}$-neighbor.  \rc{rc-w7}
 \end{lemma}

\begin{proof}
Let $v$ be a 7-vertex on three pendent triangles where $v_1$ is the neighbor of $v$ not on a pendent triangle.
If $v_1\in W_5$, then $G$ is a graph with twelve vertices, and it is easy to find an {\FII} of $G$.
Suppose $v_1\in W_{23}$, which implies that $v_1$ is a $3^-$-vertex.
We use the labels as in Figure~\ref{fig:7-vertex}.

\begin{figure}[ht]
	\centering
  \includegraphics[scale=0.75,page=11]{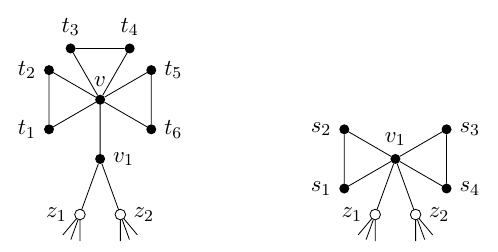} 
    \caption{An illustration for Lemma~\ref{lem:rc-w7}.}
    \label{fig:7-vertex}

\end{figure}

Let $S=\{t_1, \ldots, t_6\}$ and let $H:=G-(S\cup\{v\})$.
By \eqref{eqref:basic:rho} and \eqref{eqref:basic:rho2},  $\rho^*_H(v_1)\ge -4\cdot 7+3\cdot 10=2$.
Then $\mad(H')\leq {8\over 3}$, where $H'$ be the graph obtained from $H$ by attaching two pendent triangles to $v_1$.
Since $|V^*(H')|< |V^*(G)|$, by the minimality of $G$,   $H'$ has  an {\FII}, which also gives an {\FII}
\fii{F}{I_\alpha}{I_\beta} of $H$.
If $v_1\in I_\alpha$, then since \fii{(F+S+v_1)}{(I_\alpha-v_1+v)}{I_\beta} is not an {\FII} of $G$, we know $z_1, z_2\in F$.
Now, \fii{(F+S)}{I_\alpha}{(I_\beta+v)} is an {\FII} of $G$.
If $v_1\in F$, then by considering two pendent triangles of $H'$, we have $z_1, z_2\in F$, and so \fii{(F+S)}{(I_\alpha+v)}{I_\beta} is an {\FII} of $G$.
\end{proof}

\section{Reducible Configurations [C${}^{\prime}$1]-[C${}^{\prime}$5]}\label{sec:proof:configuration2}

\begin{lemma}\label{lem:rc-cycle-w3-w2}
The subgraph of $G$ induced by $W_{23}$ is a forest. \rcp{rcp-cycle-w3-w2}
\end{lemma}
\begin{proof}
Suppose to the contrary that there is a cycle $C$ consisting of $W_{23}$-vertices.
By~\rc{rc-22-c}~and~\rc{rc-w3w3}, $C$ is an even cycle such that a $W_3$-vertex and a $W_2$-vertex appear alternatively.
Let $C:u_1v_1u_2v_2\ldots u_kv_k$ ($k\ge 2$) where $u_i\in W_2$ and $v_i\in W_3$.
Let $z_i$ be the neighbor of $v_i$ not on $C$, and let $Z=\{z_1,\ldots,z_k\}$.
Let $H=G-V(C)$.
Since $|V^*(H)|< |V^*(G)|$, by the minimality of $G$, $H$ has an {\FII}.

\begin{claim}\label{claim:w2-w3-cycle}
For every {\FII} \fii{F}{I_\alpha}{I_\beta} of $H$,
either $Z\subset F$ or $Z\cap F=\emptyset$.
\end{claim}
\begin{proof}
Let $Z\cap F=\{z_{i_1}, z_{i_2},\ldots,z_{i_t}\}$ where $i_1<\cdots<i_t$.
Suppose to the contrary that $1\leq t\leq k-1$.
Without loss of generality, assume $z_k\in I_{\alpha}$ and $i_1=1$ so that $z_{i_1}=z_1\in F$.
If $t=1$, then  \fii{(F+V(C)-u_1)}{I_{\alpha}}{(I_\beta+u_1)} is an {\FII} of $G$.
Now assume $t\ge 2$.
Add $u_1$ to $I_\beta$.
For each $s\in\{2, 3, \ldots, t\}$, add $u_{i_s}$ to either $I_\alpha$ or $I_\beta$ one by one according to the following rule:
add $u_{i_s}$ to $I_\alpha$ and $I_\beta$ if either $u_{i_s-1}$ or $z_{i_s-1}$ is in $I_\beta$ and $I_\alpha$, respectively.
Note that both $u_{i_s-1}\in I_\alpha$, $z_{i_s-1}\in I_\beta$ and $u_{i_s-1}\in I_\beta$, $z_{i_s-1}\in I_\alpha$ cannot happen since $u_{i_s-1}$ is added to either $I_\alpha$ or $I_\beta$ if and only if $z_{i_s-1}\in F$.
Also, since $t\le k-1$, $u_{i_t}$ and $u_1$ has distance at least three.
Now add all vertices in $V(C)\setminus \{u_{i_1},\ldots, u_{i_t}\}$ to $F$, which results in an  {\FII} of $G$.
\end{proof}

\begin{claim}\label{claim:only-w2-w3}
Let $H'$ be a graph with an {\FII} \fii{F'}{I'_{\alpha}}{I'_\beta}.
\begin{itemize}
\item[\rm(i)] If $H'$ is the graph obtained from  $H$ by attaching a pendent triangle $T$ to some $z_i$, then $Z\cap F'=\emptyset$.
\item[\rm(ii)] If $H'$ is the graph obtained from  $H$ by attaching $J_1$ in Figure~\ref{fig:2B_2} to some $z_i$, then  $Z\subset F'$.
\end{itemize}
\end{claim}
\begin{proof}
Let \fii{F}{I_\alpha}{I_\beta} be a restriction of \fii{F'}{I'_{\alpha}}{I'_\beta} to $V(H)$, which is an {\FII} of $H$.
By Claim~\ref{claim:w2-w3-cycle}, either $Z\subset F$ or $Z\cap F=\emptyset$.
Hence, either $Z\subset F'$ or $Z\cap F'=\emptyset$.

\noindent (i) Suppose to the contrary that $Z\cap F'\neq \emptyset$, so $Z\subset F'$.
Without loss of generality, let $z_i=z_1$ and assume a 2-vertex on $T$ belongs to $I'_\alpha$.
Let $U_1=\{u_i\mid i\text{ is odd }, 3\le i\le k\}$
and $U_2=\{u_i\mid i\text{ is even }, 4\le i\le k\}$.
Then  \fii{(F+v_2v_3\ldots v_k+u_1u_2)}{(I_{\alpha}+U_1)}{(I_\beta+U_2+v_1)} is an {\FII} of $G$.

\smallskip

\noindent (ii) Suppose to the contrary that $Z\not\subset F'$, so $Z\cap F'=\emptyset$.
Without loss of generality assume $z_1\in I'_{\alpha}$.
By Lemma~\ref{lema:2B2} (i), the center of $J_1$ is in $I'_{\beta}$, so $G$ has an {\FII}
\fii{(F+V(C)-v_1)}{I_{\alpha}}{(I_\beta+v_1)}.
\end{proof}

\noindent(Case 1)
Suppose $\rho^*_H(z_i)\ge 2$ for some $i$. Let $H''$ be the graph obtained from  $H$ by attaching a pendent triangle $T$ and the graph $J_1$ in Figure~\ref{fig:2B_2} to $z_i$.
Then $\mad(H'')\leq {8\over 3}$.
Since $|V^*(H'')|<|V^*(G)|$, by the minimality of $G$, ${H''}$ has an {\FII}.
 Note that we may apply both Claim~\ref{claim:only-w2-w3} (i) and (ii) to $H''$ and conclude both $Z\cap F'=\emptyset$ and $Z\subset F$, which is a contradiction.

\smallskip

\noindent(Case 2)
Suppose $\rho^*_H(z_i)\le 1$ for all $i\in[k]$.
By \eqref{eqref:basic:rho} and \eqref{eqref:basic:rho2},
$\sum_{i=1}^{k} \rho^*_H{(z_i)}\ge \rho^*_H(Z)\ge  -4\cdot 2k+ 3\cdot 3k \ge k$,
and so $\rho^*_H(z_i)=1$ for all $i$.
Then we have $ \rho^*_H(z_1z_2)\ge 2$, since
\[\rho^*_H(z_{1}z_2)+k-2=\rho^*_H(z_1z_2)+\sum_{i\ge 3}\rho^*_H(z_i)\ge \rho^*_H(Z)\ge k.\]
Let $H''$ be the graph obtained from $G$ by attaching a pendent triangle $T$ to $z_1$ and attaching the graph $J_1$ in Figure~\ref{fig:2B_2} to $z_2$.
As in the previous case, $\mad(H'')\leq {8\over 3}$.
Since $|V^*(H'')|<|V^*(G)|$, by the minimality of $G$,   ${H''}$ has an {\FII}.
Note that we may apply both Claim~\ref{claim:only-w2-w3} (i) and (ii) to $H''$ and conclude both $Z\cap F'=\emptyset$ and $Z\subset F$, which is a contradiction.
\end{proof}
The following lemma implies \rcp{rcp-4-only2}.
\begin{lemma}\label{lemma:rc-4-only2}
In $G$, there is no $V_4$-vertex $v$ satisfying one of the following:
\begin{itemize}
\item[\rm (i)] $v$ has two $W_5$-neighbors. 
\item[\rm(ii)] $v$ has three $W_2$-neighbors and the last neighbor is in $W_{235}$.
\item[\rm(iii)] $v$ has two $W_2$-neighbors and the other two neighbors are in $W_{235}$.
\end{itemize}
\end{lemma}

\begin{proof}
Let $u_1, u_2, u_3,u_4$ be neighbors of a vertex $v\in V_4$.

\smallskip

\noindent (i) Suppose to the contrary that  $u_1,u_2\in W_5$.
Let $S=N_G[u_1]\cup N_G[u_2]$ and $H=G-S$.
Let $H'$ be the graph obtained from $H$ by adding an edge $u_3u_4$.
By \eqref{eqref:basic:rho} and \eqref{eqref:basic:rho2}, we have  $\rho^*_H(u_3u_4)\ge -4\cdot 11+3\cdot 16=4$, and so $\mad(H')\le \frac{8}{3}$.
Since $|V^*(H')|<|V^*(G)|$, by the minimality of $G$, $H'$ has an {\FII} \fii{F}{I_\alpha}{I_\beta}, which is also an {\FII} of $H$.
From each $u_i$ where $i\in[2]$, let $t_i, t'_i$ be 2-vertices from different pendent triangles on $u_i$.
Now, \fii{(F+S-t_1t'_1t_2t'_2)}{(I_\alpha+t_1t_2)}{(I_\beta+t_1't'_2)} is an {\FII} of $G$.

\smallskip

\noindent (ii) Suppose to the contrary that $u_1, u_2, u_3\in W_2$ and $u_4\in W_{235}$.
We use the labels as in  Figure~\ref{fig:rc-4-only}.
Let $H=G-S$, where
\[ S=\begin{cases}
\{v,u_1,u_2,u_3\}\cup N_G[u_4] & \text{if }u_4\in W_5\\
\{v,u_1,u_2,u_3,x_4,x'_4\}& \text{if  }u_4\in W_3\\
\{v,u_1,u_2,u_3,u_4\}&\text{if }u_4\in W_2.
\end{cases}\]
\begin{figure}[ht]
	\centering
  \includegraphics[scale=0.75,page=12]{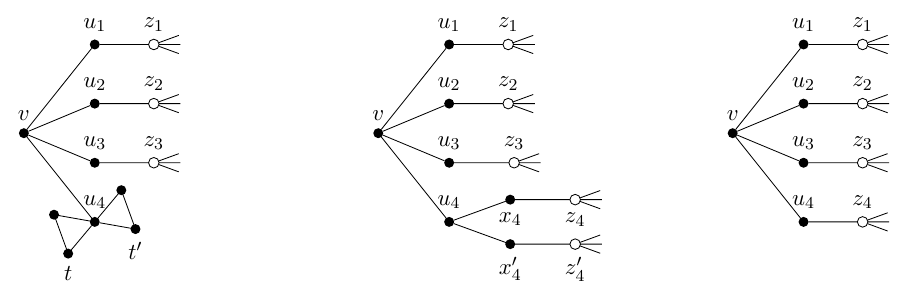}  
  \caption{An illustration of Lemma~\ref{lemma:rc-4-only2} (ii).}
  \label{fig:rc-4-only}

\end{figure}
By the minimality of $G$, $H$ has an {\FII} \fii{F}{I_\alpha}{I_\beta}.
If $u_4\in W_5$, then since \fii{(F+S-tt')}{(I_\alpha+t)}{(I_\beta+t')} is not an {\FII} of $G$,  at least two  $z_i$'s are in  $F$.
Then either
\fii{(F+S-vu_4)}{(I_\alpha+v)}{(I_\beta+u_4)} or
\fii{(F+S-vu_4)}{(I_\alpha+u_4)}{(I_\beta+v)} is an {\FII} of $G$.

Suppose that $u_4\in W_3$.
If at most one of $z_1$, $z_2$, $z_3$ is in $F$,
then
\fii{(F+S-u_4)}{I_\alpha}{(I_\beta+u_4)},
\fii{(F+S-u_4)}{(I_\alpha+u_4)}{I_\beta}, or
\fii{(F+S)}{I_\alpha}{I_\beta}
is an {\FII} of $G$.
If two of $z_1,z_2,z_3$ are in $F$,
then
\fii{(F+S-v)}{(I_\alpha+v)}{I_\beta},
\fii{(F+S-v)}{I_\alpha}{(I_\beta+v)},
\fii{(F+S-vu_4)}{(I_\alpha+v)}{(I_\beta+u_4)}, or
\fii{(F+S-vu_4)}{(I_\alpha+u_4)}{(I_\beta+v)}
is an {\FII} of $G$.

Now suppose that $u_4\in W_2$.
\begin{claim}\label{claim:rc-4-F2}
For every {\FII} \fii{F}{I_{\alpha}}{I_{\beta}} of $H$, exactly two of $z_i$'s are in $F$. 
\end{claim}
\begin{proof}
If at most one $z_i$ is in $F$,
then $G$ has an {\FII} \fii{(F+S)}{I_{\alpha}}{I_{\beta}}.
If at least three $z_i's$ are in $F$,
then either \fii{(F+ S-v)}{(I_{\alpha}+v)}{I_{\beta}} or
\fii{(F+ S-v)}{I_{\alpha}}{(I_{\beta}+v)} is  an {\FII} of $G$.
\end{proof}

By~\eqref{eqref:basic:rho2}, we  have $\rho_H^*(z_1z_2z_3z_4)\ge 4$.
\begin{claim}\label{claim:please_this_is_the_last2}\begin{itemize}
\item[\rm(a)]$\rho^*_H(z_i)\le 2$ for every $i\in [4]$ where $i\neq j$.
\item [\rm(b)]
$\rho^*_H(z_iz_j)\le 3$ for every $i,j\in [4]$ where $i\neq j$.
\item[\rm(c)] If $\rho^*_H(z_iz_j)=3$ for $i,j\in[4]$ where $i\neq j$, then
$\rho^*_H(z_iz_jz_k)=3$ for every $k\in[4]\setminus\{i,j\}$.
\item[\rm(d)]
There are no distinct $i,j,k\in[4]$ such that attaching a pendent triangle at each of $z_i, z_j, z_k$ results in a graph $H'$ satisfying $\mad(H')\le \frac{8}{3}$. 
\end{itemize}
\end{claim}

\begin{proof} In the proof of each case, we define a graph $H'$ by modifying $H$ so that $\mad(H')\le \frac{8}{3}$ and $|V^*(H')|<|V^*(G)|$.
By the minimality of $G$, $H'$ has an {\FII} \fii{F'}{I'_\alpha}{I'_\beta}.
Let $F=F'\cap V(H)$, $I_\alpha=I'_\alpha\cap V(H)$, and $I_\beta=I'_\beta\cap V(H)$.

 (a) Suppose to the contrary that $\rho^*_H(z_1)\ge 3$.
Let $H'$ be the graph obtained from $H$ by attaching $J_2$ and a pendent triangle to $z_1$.
Note that since $\rho^*_H(z_1)\ge 3$ we know $\mad(H')\le\frac{8}{3}$.
By Lemma~\ref{lema:2B2} (ii), $z_1\in F$.
By considering a pendent triangle $T_1$, together with Claim~\ref{claim:rc-4-F2}, either
\fii{(F+S-u_1)}{(I_{\alpha}+u_1)}{I_{\beta}} or
\fii{(F+S-u_1)}{I_{\alpha}}{(I_{\beta}+u_1)} is an {\FII} of $G$.

 (b) Suppose to the contrary that $\rho^*_H(z_{3}z_4)\geq4$.
Since $\rho^*_H(z_3)+\rho^*_H(z_4)\ge \rho^*_H(z_{3}z_4)$ by \eqref{eqref:basic:rho},
we may assume that $\rho^*_H(z_3)\ge 2$.
Let $H'$ be the graph obtained from $H$ by adding an edge $z_3z_4$ and attaching $J_1$ to $z_3$.
Since $\rho^*_H(z_3z_4)\ge 4$ and $\rho^*_H(z_3)\ge 2$, we know $\mad(H')\le\frac{8}{3}$.
Suppose $z_3\in F$.
By Claim~\ref{claim:rc-4-F2}, exactly one of  $z_1,z_2,z_4$ is in $F$.
If $z_4\in F$, then \fii{(F+S)}{I_{\alpha}}{I_{\beta}} is an {\FII} of $G$.
If $z_4\not\in F$, then either \fii{(F+S-u_3)}{(I_{\alpha}+u_3)}{I_{\beta}} or
\fii{(F+S-u_3)}{I_{\alpha}}{(I_{\beta}+u_3)} is an {\FII} of $G$.
Now suppose $z_3\not\in F$, say $z_3\in I_\alpha$.
By Lemma~\ref{lema:2B2} (i), by considering the center of $J_1$, we know $z_4\in F$.
Now, $G$ has an {\FII} \fii{(F+S-u_4)}{(I_{\alpha}+u_4)}{I_{\beta}}.

 (c) Suppose to the contrary that $\rho^*_H(z_3z_4)=3$ and $\rho^*_H(z_1z_3z_4)\ge 4$.
By \eqref{eqref:basic:rho}, $\rho^*_H(z_{1})+\rho^*_H(z_3z_4)\ge \rho^*_H(z_1z_3z_4)\ge 4$ and so $\rho^*_H(z_{1})\ge 1$.
Let $H'$ be the graph obtained from $H$ by adding a pendent triangle $T_1$ at $z_1$ and an edge $z_3z_4$.
If $z_4\in F$, then since \fii{(F+S)}{I_{\alpha}}{I_{\beta}} is not an {\FII} of $G$, we may assume $z_3\in I_{\alpha}$.
Now, \fii{(F+S-u_4)}{(I_{\alpha}+u_4)}{I_{\beta}} is an {\FII} of $G$.
If $z_3,z_4\not\in F$, then $z_1\in F$.
By considering a pendent triangle  $T_1$, either
\fii{(F+S-u_1)}{(I_{\alpha}+u_1)}{I_{\beta}} or
\fii{(F+S-u_1)}{I_{\alpha}}{(I_{\beta}+u_1)} is an {\FII} of $G$.

 (d) Suppose to the contrary that attaching a pendent triangle $T_i$ to each of $z_1,z_2,z_3$ results in a graph $H'$ satisfying $\mad(H')\le \frac{8}{3}$.
By Claim~\ref{claim:rc-4-F2}, we may assume $z_1\in F$ and a 2-vertex on $T_1$ is in $I_{\alpha}$.
Now, \fii{(F+S-u_1)}{(I_{\alpha}+u_1)}{I_{\beta}} is an {\FII} of $G$.
\end{proof}

If $\rho^*_H(z_i)=0$  for two integers $i\in[4]$, say $\rho^*_H(z_1)=\rho^*_H(z_2)=0$,
then by \eqref{eqref:basic:rho} and \eqref{eqref:basic:rho2},
$\rho^*_H(z_1)+\rho^*_H(z_2)+\rho^*_H(z_3z_4) \ge \rho^*_H(z_1z_2z_3z_4)\ge 4$.
This implies $\rho^*_H(z_3z_4)\ge 4$, which is a contradiction to Claim~\ref{claim:please_this_is_the_last2} (b).
Hence, we may assume that $\rho^*_H(z_i)\ge 1$ for $i\in[3]$.
If $\rho^*_H(z_1z_2), \rho^*_H(z_2z_3), \rho^*_H(z_1z_3)\leq 1$,
then $3\geq \rho^*_H(z_1z_2)+\rho^*_H(z_2z_3)+\rho^*_H(z_1z_3)\ge 2\rho^*_H(z_1z_2z_3)$, which implies  $\rho^*_H(z_1z_2z_3)=1$.
Therefore $\rho^*_H(z_4)\ge 3$, which is a contradiction to Claim~\ref{claim:please_this_is_the_last2} (a).
Hence, we may assume $\rho^*_H(z_1z_2)\ge 2$.
If $\rho^*_H(z_1z_2z_3)\le 2$, then by~\eqref{eqref:basic:rho}
\begin{eqnarray*}
&&\rho^*_H(z_1z_2z_4)\ge \rho^*_H(z_1z_2z_3z_4)+\rho^*_H(z_1z_2)-\rho^*_H(z_1z_2z_3) \ge 6-2=4,\\
&& \rho^*_H(z_1z_4),\rho^*_H(z_2z_4)\ge \rho^*_H(z_4)\ge \rho^*_H(z_1z_2z_3z_4)-\rho^*_H(z_1z_2z_3)\ge 2.
\end{eqnarray*}
Thus, attaching one pendent triangle at each $z_i$ for $i\in\{1,2,4\}$ results in a graph $H'$ with $\mad(H')\le \frac{8}{3}$, which is a contradiction to Claim~\ref{claim:please_this_is_the_last2} (d).
Hence, $\rho^*_H(z_1z_2z_3)\ge 3$.
Note that by \eqref{eqref:basic:rho}
\begin{eqnarray}\label{eq:last:please}
&&\rho^*_H(z_{2}z_{3})+\rho^*_H(z_{1}z_{3})\ge \rho^*_H(z_{1}z_{2}z_{3})+\rho^*_H(z_3)\ge 3+1=4.
\end{eqnarray}
If $\rho^*_H(z_2z_3)=2$, then  $\rho^*_H(z_1z_{3})\ge 2$ by \eqref{eq:last:please}.
Thus, attaching one pendent triangle at each $z_i$ for $i\in\{1,2,3\}$ results in a graph $H'$ with $\mad(H')\le \frac{8}{3}$, which is a contradiction to Claim~\ref{claim:please_this_is_the_last2} (d).
Hence  $\rho^*_H(z_2z_3)\neq 2$, and by symmetry,  $\rho^*_H(z_1z_3)\neq 2$.
We may assume $\rho^*_H(z_2z_3)=1$.
Note that $\rho^*_H(z_1z_3z_4)\ge \rho^*_H(z_1z_3)=3$ where the second equality is from \eqref{eq:last:please} and  Claim~\ref{claim:please_this_is_the_last2} (b).
Together with Claim~\ref{claim:please_this_is_the_last2} (d), we have $\rho^*_H(z_1z_2z_3)=3$, and therefore by~\eqref{eqref:basic:rho}
\begin{eqnarray*}
&&\rho^*_H(z_4)\ge \rho^*_H(z_{1}z_{2}z_{3}z_{4})-\rho^*_H(z_{1}z_{2}z_{3})\ge 1,\\
&& \rho^*_H(z_{1}z_{4})\ge \rho^*_H(z_{1}z_{2}z_{3}z_{4})-\rho^*_H(z_{2}z_{3})\ge 3,\\
&&\rho^*_H(z_{3}z_{4})\ge \rho^*_H(z_1z_2z_3z_4)+\rho^*_H(z_3)-\rho^*_H(z_1z_2z_3)\ge 4+1-3=2.
\end{eqnarray*}
Thus, attaching one pendent triangle at each $z_i$ for $i\in\{1,3,4\}$ results in a graph $H'$ satisfying $\mad(H')\le \frac{8}{3}$, which is a contradiction to Claim~\ref{claim:please_this_is_the_last2} (d).

\smallskip

\noindent (iii) Suppose to the contrary that $u_1,u_2\in W_2$ and $u_3, u_4\in W_{235}$.
By (i) and (ii),
we may assume that $u_3\in W_3$ and $u_4\in W_{35}$.
We use the labels as in  Figure~\ref{fig:rc-4-only2-2}.
Let $S=N_G[v]\cup N_G[u_3]\cup N_G[u_4]$.
By the minimality of $G$,  $H=G-S$ has an {\FII}.
Let $Z$ be the set of all $z_i$'s and $z'_i$'s.

\begin{figure}[ht]
	\centering
  \includegraphics[scale=0.75,page=13]{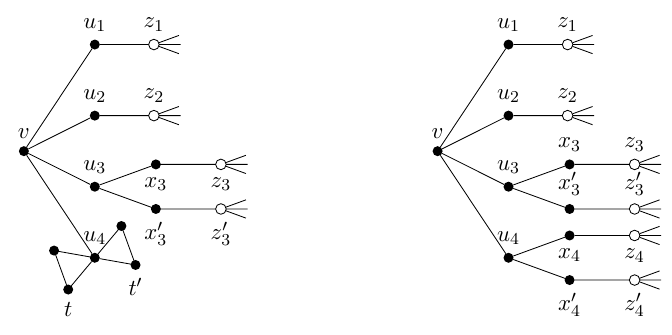} 
    \caption{An illustration for Lemma~\ref{lemma:rc-4-only2} (iii).}
  \label{fig:rc-4-only2-2}

\end{figure}

\begin{claim}\label{subclaim:rc-4}
For every {\FII} \fii{H}{I_\alpha}{I_\beta} of $H$, $Z\subset F$.
\end{claim}
\begin{proof}
Suppose to the contrary that $Z\not\subset F$.
Suppose $u_4\in W_5$.
If $z_1\not\in F$, say $z_1\in I_\alpha$, then
since neither \fii{(F+S-u_3u_4)}{(I_\alpha+u_3)}{(I_\beta+u_4)} nor
\fii{(F+S-u_3u_4)}{(I_\alpha+u_4)}{(I_\beta+u_3)}
is an {\FII} of $G$, we may assume $z_3\in I_\alpha$ and $z'_3\in I_\beta$.
Yet, \fii{(F+S-tt')}{(I_\alpha+t)}{(I_\beta+t')} is an {\FII} of $G$.
Therefore, $z_1, z_2\in F$, and since $Z\not\subset F$, we know $\{z_3, z'_3\}\not\subset F$.
Now, \fii{(F+S-vu_4)}{(I_\alpha+v)}{(I_\beta+u_4)} is an {\FII} of $G$.

Suppose that $u_4\in W_3$.
If $z_3\not\in F$, then we may assume $z_3\in I_\alpha$.
Since neither \fii{(F+S-v)}{(I_\alpha+v)}{I_\beta} nor \fii{(F+S-vu_4)}{(I_\alpha+v)}{(I_\beta+u_4)} is an {\FII} of $G$, we may assume $z_1\in I_\alpha$.
Similarly, we conclude $z_2\in I_\beta$.
Now, \fii{(F+S)}{I_\alpha}{I_\beta}, \fii{(F+S-u_4)}{(I_\alpha+u_4)}{I_\beta}, or \fii{(F+S-u_4)}{I_\alpha}{(I_\beta+u_4)} is an {\FII} of $G$.
Therefore, $z_3, z'_3, z_4, z'_4\in F$.
Since $Z\not\subset F$, we know $\{z_1, z_2\}\not\subset F$.
Now, \fii{(F+S-u_3u_4)}{(I_\alpha+u_3)}{(I_\beta+u_4)} is an {\FII} of $G$.
\end{proof}

Note that by \eqref{eqref:basic:rho2}, if $u_4\in W_5$, then
$\rho^*_H(Z)\ge -4\cdot 11 +3\cdot 16=4=|Z|$, and if $u_4\in W_3$, then
$\rho^*_H(Z)\ge -4\cdot 9 +3\cdot 14=6=|Z|$.

Suppose  $\rho^*_H(z_1z_2)\ge 3$.
Let $H'$ be the graph obtained from $H$ by adding an edge $z_1z_2$.
Then $\mad(H')\leq {8\over 3}$.
Since $|V^*(H')|<|V^*(G)|$, by the minimality of $G$, ${H'}$ has an {\FII}, which also gives an {\FII} \fii{F}{I_\alpha}{I_\beta} of $H$.
By Claim~\ref{subclaim:rc-4},  $Z\subset F$, so $G$ has an {\FII} \fii{(F+S-u_3u_4)}{(I_\alpha+u_3)}{(I_\beta+u_4)}.

Now suppose $\rho^*_H(z_1z_2)\le 2$.
Then $\displaystyle\!\!\!\!\sum_{z\in Z\setminus\{z_1,z_2\}} \!\!\!\!\!\!\!\rho^*_H(z)\ge |Z|-2$, since $\displaystyle \rho^*_H(z_{1}z_{2})+\!\!\!\!\!\!\sum_{z\in Z\setminus\{z_1,z_2\}} \!\!\!\!\!\!\!\rho^*_H(z)\ge |Z|$ by \eqref{eqref:basic:rho}.
Without loss of generality assume $\rho^*_H({z_3})\ge 1$.
Let $H'$ be the graph obtained from $H$ by attaching a pendent triangle $T$ to $z_3$.
Then $\mad(H')\leq {8\over 3}$.
Since $|V^*(H')|<|V^*(G)|$, by the minimality of $G$,   ${H'}$ has an {\FII}, which also gives an {\FII} \fii{F}{I_\alpha}{I_\beta} of $H$.
By Claim~\ref{subclaim:rc-4}, we have $Z\subset F$.
By considering the pendent triangle $T$, either  \fii{(F+S-vu_4x_3)}{(I_\alpha+x_3u_4)}{(I_\beta+v)} or
\fii{(F+S-vu_4x_3)}{(I_\alpha+v)}{(I_\beta+x_3u_4)} is an {\FII} of $G$.
\end{proof}

\begin{lemma}\label{lem:rc-6-two-pendent}
In $G$, there is no 6-vertex on two pendent triangles with a $W_{235}$-neighbor and a different $W_{25}$-neighbor. \rcp{rcp-6-two-pendent}
\end{lemma}
\begin{proof}
Suppose to the contrary that
there is a 6-vertex $v$ on exactly two pendent triangles with a $W_{25}$-neighbor $u_1$  and a different $W_{235}$-neighbor $u_2$.
If $u_2\in W_5$, then by considering an {\FII} of $G-(N_G[v]\cup N_G[u_2])$, it is easy to find an {\FII} of $G$.
Assume $u_2\in W_{23}$.
We use the labels as in  Figure~\ref{fig:rc-6p}.
\begin{figure}[ht]
	\centering
  \includegraphics[scale=0.75,page=14]{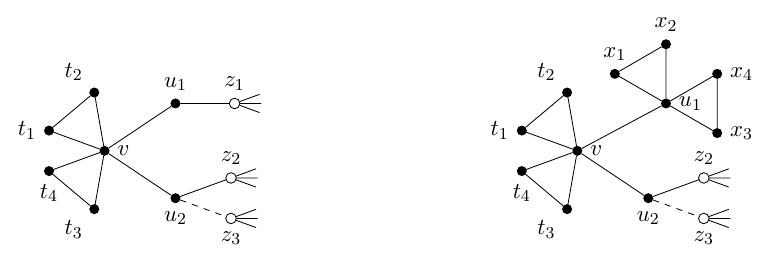} 
  \caption{An illustration of Lemma~\ref{lem:rc-6-two-pendent}.}
  \label{fig:rc-6p}

\end{figure}

If $u_1\in W_2$, then let $S=N_G[v]$, and if $u_1\in W_5$, then let $S=N_G[v]\cup N_G[u_1]$.
By the minimality of $G$, $H=G-S$ has an {\FII}.
When $u_2$ is a $2$-vertex, we ignore $z_3$ in Claim~\ref{claim:z2z3}.
\begin{claim}\label{claim:z2z3}
For every {\FII} \fii{F}{I_{\alpha}}{I_\beta} of $H$, $\{z_2,z_3\}\subset F$.
\end{claim}
\begin{proof}
Suppose to the contrary that  $\{z_2,z_3\}\not\subset F$.
Without loss of generality, assume $z_2\in I_\alpha$.
If $u_1\in W_5$, then \fii{(F+S-t_1t_3x_1x_3)}{(I_\alpha+t_1x_1)}{(I_\beta+t_3x_3)} is an {\FII} of $G$.
If $u_1\in W_2$, then since \fii{(F+S-t_1t_3)}{(I_\alpha+t_1)}{(I_\beta+t_3)} is not an {\FII} of $G$, we conclude $z_1, z_3\in F$.
Now, \fii{(F+S-v)}{I_\alpha}{(I_\beta+v)} is an {\FII} of $G$.
\end{proof}

Suppose $u_2\in W_2$. Let  \fii{F}{I_{\alpha}}{I_\beta} be an {\FII} of $H$.
By Claim~\ref{claim:z2z3}, $z_2\in F$.
If $u_1\in W_5$,
then \fii{(F+S-vu_1)}{(I_{\alpha}+u_1)}{(I_{\beta}+v)} is an {\FII} of $G$.
If $u_1\in W_2$,
then either \fii{(F+S-v)}{(I_{\alpha}+v)}{I_{\beta}} or \fii{(F+S-v)}{I_{\alpha}}{(I_{\beta}+v)} is an {\FII} of $G$.

Suppose  $u_2\in W_3$.
By \eqref{eqref:basic:rho2},
if $u_1\in W_2$ then $\rho^*_H(z_{2}z_{3})\ge 3$, and
if $u_1\in W_5$ then $\rho^*_H(z_{2}z_{3})\ge 4$.
Let $H'$ be the graph obtained from $H$ by adding an edge $z_2z_3$.
Then  $\mad(H')\leq {8\over 3}$.
Since $|V^*(H')|<|V^*(G)|$, by the minimality of $G$,  ${H'}$ has an {\FII} \fii{F}{I_\alpha}{I_\beta}, which is also an {\FII} of $H$.
By Claim~\ref{claim:z2z3}, $z_2,z_3\in F$.
If $u_1\in W_2$, then
either \fii{(F+S-v)}{(I_\alpha+v)}{I_\beta} or
\fii{(F+S-v)}{I_\alpha}{(I_\beta+v)} is an {\FII} of $G$.
If $u_1\in W_5$, then \fii{(F+S-vu_1)}{(I_\alpha+v)}{(I_\beta+u_1)} is an {\FII} of $G$.
\end{proof}

\begin{lemma}\label{lem:rc-5-one-pendent}
In $G$, there is no $5$-vertex $v$ on one pendent triangle with three $W_{235}$-neighbors where two are $W_2$-neighbors.
\rcp{rcp-5-one-pendent}
\end{lemma}

\begin{proof}
Let $v$ be a $5$-vertex on one pendent triangle
with three $W_{235}$-neighbors where two are $W_2$-neighbors.
We use the labels as in  Figure~\ref{fig:rcp-5p}.
Let $H=G-S$, where  \[ S=\begin{cases}
\{v,t_1,t_2,u_1,u_2,u_3\}\cup N_G(u_3) & \text{if }u_3\in W_5\\
\{v,t_1,t_2,u_1,u_2,u_3\}& \text{if  }u_3\in W_2\\
\{v,t_1,t_2,u_1,u_2,x_3,x_4\}&\text{if }u_3\in W_3.
\end{cases}\]\begin{figure}[ht]
	\centering
  \includegraphics[scale=0.75,page=15]{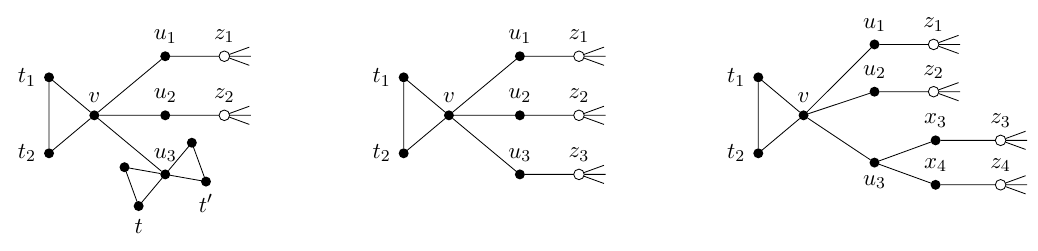} 
    \caption{An illustration for Lemma~\ref{lem:rc-5-one-pendent}.}
  \label{fig:rcp-5p}

\end{figure}
By the minimality of $G$, $H$ has an {\FII} \fii{F}{I_\alpha}{I_\beta}.
If $u_3\in W_5$, then
since neither \fii{(F+S-vu_3)}{(I_\alpha+v)}{(I_\beta+u_3)} nor
\fii{(F+S-vu_3)}{(I_\alpha+u_3)}{(I_\beta+v)} is an {\FII} of $G$,
we have $z_1,z_2\not\in F$.
Thus, $G$ has an {\FII}  \fii{(F+S-t_1tt')}{(I_\alpha+t)}{(I_\beta+t't_1)}.
If $u_3\in W_2$, then
since neither  \fii{(F+S-v)}{(I_\alpha+v)}{I_\beta}, nor
\fii{(F+S-v)}{I_\alpha}{(I_\beta+v)}
is an {\FII} of $G$, we have
  $z_1,z_2\not\in F$.
  Thus, $G$ has an {\FII}  \fii{(F+S-t_1)}{(I_\alpha+t_1)}{I_\beta}.

Now suppose that $u_3\in W_3$.
If $z_i\in F$ for some $i\in [2]$, then
we may assume that $z_1,z_2\not\in I_\alpha$ and so either
\fii{(F+S-v)}{(I_\alpha+v)}{I_\beta} or
\fii{(F+S-vu_3)}{(I_\alpha+v)}{(I_\beta+u_3)} is an {\FII} of $G$.
If $z_1,z_2\not\in F$, then
either
\fii{(F+S-t_1)}{(I_\alpha+t_1)}{I_\beta} or
\fii{(F+S-t_1u_3)}{(I_\alpha+t_1)}{(I_\beta+u_3)} is an {\FII} of $G$.
\end{proof}

\begin{lemma}\label{lem:rc-cycle-3-d4-same2}
In $G$, there is no $W_3$-vertex $u$ with a $3$-neighbor such that a $2$-neighbor of $u$ has only $3^-$-neighbors.
\end{lemma}
\begin{proof}
Suppose to the contrary that there is a vertex $u\in W_3$ with a $3$-neighbor $z_1$ and a 2-neighbor
$x_2$ with only $3^-$-vertices.
We use the label as in  Figure~\ref{fig:sc2:lem2}.
By Lemmas
~\ref{lem:rc-w3''}
and~\ref{lem:rc-cycle-w3-w2},
 all $z_i$'s are distinct.
Let $S=\{u,x_2,x_3\}$, and $H=G-S$.
\begin{figure}[ht]
	\centering
  \includegraphics[scale=0.75,page=16]{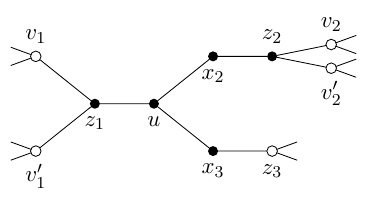} 
\caption{An illustration for Lemma~\ref{lem:rc-cycle-3-d4-same2}}
  \label{fig:sc2:lem2}

\end{figure}

\begin{claim}\label{claim:c'2-claim}
For every {\FII} \fii{F}{I_\alpha}{I_\beta} of $H$,
$z_1,z_3\in F$ and $z_2\not\in F$.
\end{claim}
\begin{proof}
If at most one $z_i$ is in $F$, then
\fii{(F+S)}{I_\alpha}{I_\beta} is an {\FII} of $G$.
Suppose that $Z\subset F$.
Since neither
\fii{(F+x_2x_3)}{(I_\alpha+u)}{I_\beta} nor
\fii{(F+x_2x_3)}{I_\alpha}{(I_\beta+u)}
is an {\FII} of $G$, we may assume $v_1\in I_\alpha$ and $v'_1\in I_\beta$.
Since \fii{(F+S)}{I_\alpha}{I_\beta} is not an {\FII} of $G$, either $v_2$ or $v'_2$ is in $F$.
Now, either
\fii{(F+ux_3)}{(I_\alpha+x_2)}{I_\beta} or
\fii{(F+ux_3)}{I_\alpha}{(I_\beta+x_2)} is an {\FII} of $G$.
Hence, exactly two of $z_i$'s are in $F$.
Suppose to contrary that $z_2\in F$.
Since none of \fii{(F+S)}{I_\alpha}{I_\beta},
\fii{(F+ux_3)}{(I_\alpha+x_2)}{I_\beta},
or \fii{(F+ux_3)}{I_\alpha}{(I_\beta+x_2)}
is an {\FII}  of $G$,
we may assume that $z_1\in I_\alpha$, $v_2\in I_\beta$, and $v'_2\in F$.
Since
\fii{(F+x_2x_3)}{I_\alpha}{(I_\beta+u)}
is not an {\FII}  of $G$, either $v_1$ or $v'_1$ is in $I_\beta$.
Now, \fii{(F+z_1x_2x_3)}{(I_\alpha+u-z_1)}{I_\beta}
is an {\FII} of $G$.
\end{proof}

Suppose $\rho^*_H(z_2)\ge 2$.
Let $H'$ be the graph obtained from $H$ by attaching $J_2$ to $z_2$.
Since $|V^*(H')|< |V^*(G)|$, by the minimality of $G$, $H'$ has an {\FII} \fii{F'}{I'_\alpha}{I'_\beta}.
By Lemma~\ref{lema:2B2} (ii), $z_2\in F'$, which is a contradiction to Claim~\ref{claim:c'2-claim}.
Hence, $\rho^*_H(z_2)\le 1$.

Since $\rho^*_H(z_1)+\rho^*_H(z_2)+\rho^*_H(z_3)\ge 3$ by \eqref{eqref:basic:rho} and \eqref{eqref:basic:rho2},
we have $\rho^*_H(z_1)+\rho^*_H(z_3)\ge 2$.
Suppose $\rho^*_H(z_3)\ge 1$.
Let $H'$ be the graph obtained from $H$ by attaching a pendent triangle $T$ to $z_3$.
Then $H'$ has an {\FII}, which also gives an {\FII} \fii{F}{I_\alpha}{I_\beta} of $H$.
By Claim~\ref{claim:c'2-claim}, $z_3\in F$. By considering a 2-vertex of $T$ not in $F$, we know
 either
\fii{(F+ux_2)}{(I_\alpha+x_3)}{I_\beta}
or
\fii{(F+ux_2)}{I_\alpha}{(I_\beta+x_3)}
is an {\FII} of $G$.
Hence, $\rho^*_H(z_3)=0$, and therefore $\rho^*_H(z_1)\ge 2$.

Now, let $H'$ be the graph obtained from $H$ by attaching two pendent triangles $T_1$ and $T_2$ to $z_1$.
Then $H'$ has an {\FII}, which also gives an {\FII} \fii{F}{I_\alpha}{I_\beta} of $H$.
By Claim~\ref{claim:c'2-claim}, $z_1\in F$.
By considering 2-vertices on $T_1,T_2$, we know that $v_1,v'_1\in F$.
Hence, either
\fii{(F+x_2x_3)}{(I_\alpha+u)}{I_\beta}
or \fii{(F+x_2x_3)}{I_\alpha}{(I_\beta+u)}
is an {\FII} of $G$.
\end{proof}

\begin{lemma} \label{lem:rc-cycle-3-d4}
In $G$, there is no cycle $C$ consisting of $(V_3\cup W_4)$-vertices such that every $V_3$-vertex on $C$ has a $W_{23}$-neighbor. \rcp{rcp-cycle-3-d4}
\end{lemma}

\begin{proof}
Suppose to the contrary that there is such a cycle $C:u_1u_2\ldots u_k$.
For $j\in\{2,3\}$, let \[X_j=\{u\in V(C)\cap V_3\mid \text{the neighbor of }u\text{ not on }C\text{ is a $W_j$-vertex}\}.\]
We use the labels as in the left figure of  Figure~\ref{fig:sc2:lem3}; in particular, we label the neighbors of $(X_2\cup X_3)$-vertices and their neighbors. We first consider the case where all of the $v_i$'s, $t_i$'s, and $t'_i$'s are distinct. The other case when some vertices are identical is presented afterwards.
\begin{figure}[ht]
	\centering
  \includegraphics[scale=0.75,page=17]{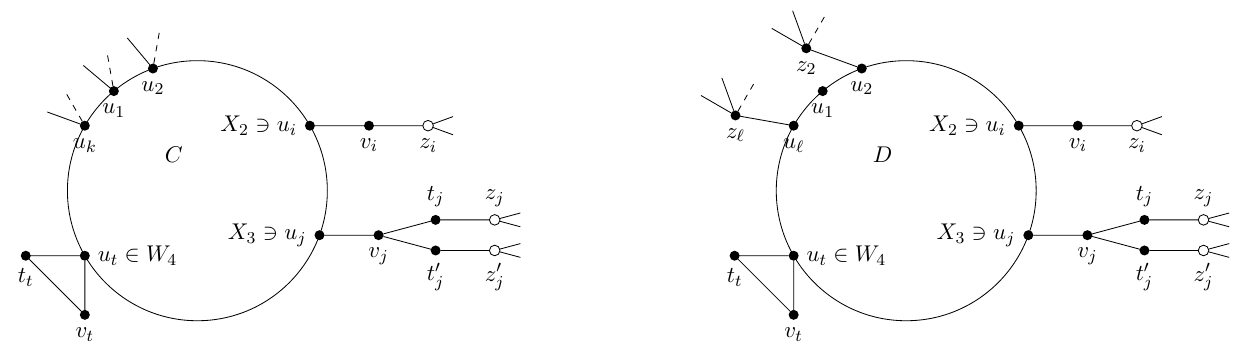} 
\caption{An illustration for Lemma~\ref{lem:rc-cycle-3-d4}}
  \label{fig:sc2:lem3}

\end{figure}

Suppose that all of the $v_i$'s, $t_i$'s, and $t'_i$'s are distinct. Let $V$ be the set of all $v_i$'s,  $T$ be the set of all $t_i$'s and $t'_i$'s, and $Z$ be the set of all $z_i$'s and $z'_i$'s. Let $S=V(C)\cup V\cup T$ and $H=G-S$.

\begin{claim}\label{claim:5cycle}
If $k=5$, then $V(C)\not\subset X_2$. 
\end{claim}

\begin{proof}
Suppose to the contrary that $k= 5$ and $V(C)=X_2$.
Since $\sum_{i}\rho_H^*(z_i)\ge -10\cdot 4+15\cdot3=5 $ by \eqref{eqref:basic:rho} and \eqref{eqref:basic:rho2}, we may assume $\rho_H^*(z_1)\ge 1$.
Let $H'$ be the graph obtained from $H$ by attaching a pendent triangle $T_1$ to $z_1$.
By the minimality of $G$, $H'$ has an {\FII}, which also gives an {\FII} \fii{F}{I_\alpha}{I_\beta} of $H$.
If $Z\subset F$, then by considering a 2-vertex on $T_1$,  either \fii{(F+S-v_1u_2u_4)}{(I_\alpha+u_2)}{(I_\beta+v_1u_4)} or
\fii{(F+S-v_1u_2u_4)}{(I_\alpha+v_1u_4)}{(I_\beta+u_2)} is an {\FII} of $G$.
If $z_i\not\in F$ for some $i$, then it is easy to find a  partition $Y_0,Y_1,Y_2$ of $V(C)\setminus\{u_i\}$ such that  \fii{(F+V+u_i+Y_0)}{(I_\alpha+Y_1)}{(I_\beta+Y_2)} is an {\FII} of $G$.
\end{proof}

By the minimality of $G$,   $H$ has an {\FII} \fii{F}{I_{\alpha}}{I_\beta}.
For $j\in\{2,3\}$, let  $X_{j}^F=\{ u_i\in X_j\mid z'_i, z_i \in F\},$ and let
$X_{j}^{\alpha\beta}=X_j \setminus X_{j}^{F}$.

\begin{claim}\label{claim:C'2:Q:notQ}
There exists an $i\in[k]$ such that $u_i\in X_{3}^{F}\cup W_4$.
\end{claim}
\begin{proof}
Suppose to the contrary that for every $i$, $u_i\in X_2\cup X_3^{\alpha\beta}$.
If $k\le 4$, then it is not hard to find a partition $Y_1,Y_2,Y_3$ of $V(C)$ such that \fii{(F+T+V+Y_1)}{(I_{\alpha}+Y_2)}{(I_{\beta}+Y_3)} is an {\FII} of $G$.
Assume  $k\ge5$.

If $u_i\in   X_2^{\alpha\beta}$ for every $i$,
then we may assume $z_1 \in I_\alpha$, and so \fii{(F+S-u_1)}{I_{\alpha}}{(I_{\beta}+u_1)} is an {\FII} of $G$.
Hence, $u_i\not\in X_2^{\alpha\beta} $ for some $i$, so $u_i\in X_2^{F}\cup X_3^{\alpha\beta}$.
We have two cases: (1)  $u_i\in  X_2^{\alpha\beta}$  and $u_j\in X_2^{F}\cup X_3^{\alpha\beta}$  for some $i,j$, and (2) $u_i\in X_2^{F}\cup X_3^{\alpha\beta}$  for every $i$.

\medskip

\noindent{(Case 1)}
Without loss of generality,  assume $u_k\in  X_2^{\alpha\beta}$ and $u_1\in X_2^{F}\cup X_3^{\alpha\beta}$.
We first find a partition of $V(G)$ by performing the following algorithm.
First, add all vertices of $X_2^{\alpha\beta}\cup T\cup V$ to $F$, and add $u_1$ to $I_{\alpha}$.
For $i\in [k-2]$, if $u_1,\ldots, u_{i}$ are determined, but $u_{i+1}$ is not yet, then do the following:
\begin{center}
If $u_i\not\in F$, then add $u_{i+1}$ to $F$.
Otherwise, for $\gamma\in \{\alpha,\beta\}$ satisfying $u_{i-1} \not\in  I_\gamma$, add $u_{i+1}$ to $I_\gamma$.
\end{center}

Note that by the algorithm, $u_k, u_2\in F$.
Since the resulting partition \fii{F}{I_{\alpha}}{I_\beta} is not an {\FII} of $G$, $u_{k-1}\in I_\alpha$, and therefore $u_{k-2}\in F$.
Since \fii{(F+u_{k-1})}{(I_\alpha-u_{k-1})}{I_\beta} must not be an {\FII} of $G$, $u_{k-2}\in X_2^{F}\cup X_3^{\alpha\beta}$.
Also, since \fii{F}{(I_\alpha-u_{k-1})}{(I_\beta+u_{k-1})}
is not an {\FII} of $G$, this implies $u_{k-3}\in I_\beta$, and therefore $u_{k-4}\in F$.
Note that this implies $k\geq 6$.
Now, \fii{(F-u_{k-2}+u_{k-1})}{(I_\alpha-u_{k-1}+u_{k-2})}{I_\beta} is an {\FII} of $G$.

\medskip

\noindent{(Case 2)}
Suppose $u_j\in X_2^{F}\cup X_3^{\alpha\beta}$  for every $j$.
If $k=5$, then by Claim~\ref{claim:5cycle}, we may assume that $u_1\in X_3^{\alpha\beta}$, and therefore
\fii{(F+S-u_2u_4)}{(I_\alpha+u_2)}{(I_\beta+u_4)},
\fii{(F+S-v_1u_2u_4)}{(I_\alpha+u_2)}{(I_\beta+v_1u_4)}, or
\fii{(F+S-v_1u_2u_4)}{(I_\alpha+v_1u_4)}{(I_\beta+u_2)} is an {\FII} of $G$.

Assume $k\neq 5$.
For $a\in\{0,1,2\}$, let  $Y_a=\{u_i\mid i\equiv a \pmod{3}\}$.
Then \fii{(F+T+V+Y_0)}{(I_{\alpha}+Y_1)}{(I_{\beta}+Y_2)} is an {\FII} of $G$ where $Y_0,Y_1,Y_2$ is a partition of $V(C)$ defined as the following:
(i) if $k\equiv 0\pmod 3$, then
no modifications to the $Y_a$'s;
(ii) if $k\equiv 1\pmod 3$, then modify the $Y_a$'s so that the last three vertices satisfy $u_{k-2},u_k\in Y_0$ and $u_{k-1}\in Y_2$;
(iii) if $k\equiv 2\pmod 3$, then modify the $Y_a$'s so that the last seven vertices satisfy $u_{k-6},u_{k-4},u_{k-2},u_k\in Y_0$, $u_{k-3}\in Y_1$, and $u_{k-1},u_{k-5}\in Y_2$.
\end{proof}

If $u_i\in X_3^{F}\cup W_4$ for every $i$, then \fii{(F+T+V(C)-u_1)}{(I_{\alpha}+V)}{(I_{\beta}+u_1)} is an {\FII} of $G$.
Hence, we assume that $u_i\not\in  X_3^{F}\cup W_4$ for some $i$.
Together with Claim~\ref{claim:C'2:Q:notQ}, we may assume $u_k\in X_2\cup X_3^{\alpha\beta}$ and $u_1\in X_3^{F}\cup W_4$.
For simplicity, let  $Q=\{v_i\in V\mid u_i\in X_3^{F}\cup W_4\}$.
We find an {\FII} of $G$ by performing the following algorithm.

\begin{itemize}
\item[Step 1.]
Add $v_1$ to $I_{\alpha}$, add $u_1, u_2$ to $F$, and add all undetermined vertices in $S-(Q\cup X_2^{F}\cup X_3^{\alpha\beta})$ to $F$.
\item[Step 2.] If vertices in $\{u_j\mid j\le i\}\cup\{v_j\mid j\le i\}$ are determined, but either $u_{i+1}\in X_2^{F}\cup X_3^{\alpha\beta}$ or $v_{i+1}\in Q$ is not determined, then do the following:
For $v_{i+1}\in Q$, add $v_{i+1}$ to exactly one of $I_{\alpha}$ or $I_\beta$ that does not contain $u_i$.
For $u_{i+1} \in X_2^{F}\cup X_3^{\alpha\beta}$,  as long as $u_i\in F$ and there is  $\gamma\in \{\alpha,\beta\}$ such that $\{u_{i-1}, u_{i}, v_i\}\cap I_\gamma=\emptyset$, add $u_{i+1}$ to  $I_{\gamma}$.
Otherwise, add $u_{i+1}$ to $F$.
\end{itemize}
Note that the resulting partition \fii{F}{I_\alpha}{I_\beta} obtained by the algorithm is not an {\FII} of $G$; the problem arises because of $u_k$.
If neither \fii{F}{I_\alpha}{I_\beta} nor \fii{F}{(I_\alpha-v_1)}{(I_\beta+v_1)} is an {\FII} of $G$, then $u_k\in F$ and $|\{v_k,u_{k-1}\}\cap F|\ge 1$.
If \fii{F}{I_\alpha}{I_\beta} is not an {\FII} of $G$, then $|\{v_2,u_3\}\cap F|\ge 1$.
Then since neither \fii{(F-u_1)}{I_\alpha}{(I_\beta+u_1)} nor \fii{(F-u_1)}{(I_\alpha-v_1+u_1)}{(I_\beta+v_1)} is an {\FII} of $G$,
we obtain $|\{v_2,u_3\}\cap F|=|\{v_k,u_{k-1}\}\cap F|=1$.
If $v_k\not \in F$, then
either \fii{(F-u_1)}{(I_\alpha+v_k)}{(I_\beta-v_k+u_1)}
or \fii{(F-u_1)}{(I_\alpha-v_1v_k)}{(I_\beta+v_ku_1)}
is an {\FII} of $G$.
Thus, $v_k\in F$, $u_{k-1}\not\in F$, and $u_k\in X_2^{F}\cup X_3^{\alpha\beta}$.
This also implies that $u_{k-1}\in X_2^{F}\cup X_3^{\alpha\beta}$ and $v_{k-1}\in F$.
If neither \fii{(F-u_k)}{(I_\alpha-v_1+u_k)}{(I_\beta+v_1)} nor \fii{(F-u_k)}{I_\alpha}{(I_\beta+u_k)} is an {\FII} of $G$, then $u_{k-2}\not\in F$.
Now, either \fii{(F+u_{k-1}-u_k)}{(I_{\alpha}-v_1u_{k-1}+u_k)}{(I_\beta+v_1)} or
\fii{(F+u_{k-1}-u_k)}{I_{\alpha}}{(I_\beta+u_k)} is an {\FII} of $G$.

\medskip

Now we consider the case where some vertices $v_i$'s, $t_i$'s, and $t'_i$'s are not distinct.
We mimic the previous case when they are all distinct, but we use a different cycle to proceed with the argument.
By Lemma~\ref{lem:rc-cycle-3-d4-same2}, for $u_i\in X_3$, we know $t_i\not\in\{v_j,t_j\}$ for some $j\neq i$.
By Lemma~\ref{lem:rc-cycle-3-d4-same1} (i), $v_i\neq v_{i+1}$.
Hence, there are two different indices $i$ and $j$ where $u_i,u_j\in X_2$, $v_i=v_j$, and the distance between $u_i$ and $u_j$ along $C$ is at least 2.
Take such $i$ and $j$ so that the distance between $u_i$ and $u_j$ along $C$ is minimum, and consider the cycle $u_iu_{i+1}\ldots u_jv_j$; we abuse notation and relabel this cycle as $D:u_1u_2\ldots u_\ell$ $(4\leq \ell\leq k)$ where $u_1$ is a $2$-vertex.
Namely, all vertices in $V(D)\setminus\{u_1\}$ are in $V_3\cup W_4$, a vertex in $(V(D)\cap V_3)\setminus\{u_2,u_\ell\}$ has a $W_{23}$-neighbor, and the neighbors of $u_2$ and $u_{\ell}$ not on $D$ are in $V_3\cup W_4$.
Let $z_2$ and $z_{\ell}$ be the neighbor of $u_2$ and $u_{\ell}$, respectively, not on $D$.
See the right figure of Figure~\ref{fig:sc2:lem3} for an illustration.
Redefine the following sets: $V=\{v_i \mid u_i\in V(D)\}$, $T=\{t_i,t'_i \mid u_i\in V(D)\}$, $Z=\{z_i,z'_i \mid u_i\in V(D)\}$.
Also, restrict $X_j$ to be $X_j\cap V(D)$.
Note that by the choice of $D$, all of the $v_i$'s, $t_i$'s, and $t'_i$'s are distinct.

Let $S=V(D)\cup V\cup T$. By the minimality of $G$, the graph $H=G-S$ has an {\FII} \fii{F}{I_{\alpha}}{I_\beta}.
For $j\in\{2,3\}$, let  $X_{j}^F=\{ u_i\in X_j\mid z'_i, z_i \in F\},$ and let
$X_{j}^{\alpha\beta}=X_j \setminus X_{j}^{F}$.
For simplicity, let  $Q=\{v_i\in V\mid u_i\in X_3^{F}\cup W_4\}$.
We attempt to find an {\FII}  of $G$ by performing the following algorithm.
\begin{itemize}
\item[Step 1.] Add $u_1,u_2,u_\ell$ to $F$.
If $z_2\not\in F$, then $u_3\in F$.
(If $z_2\in F$, then leave $u_3$ undetermined.)
Add all undetermined vertices in $S-(Q\cup X_2^{F}\cup X_3^{\alpha\beta})$ to $F$.
\item[Step 2.] Same as Step 2 of the previous case.
\end{itemize}
Let \fii{F}{I_\alpha}{I_\beta} be a resulting partition by the algorithm.
Note that we had a choice to choose either $I_\alpha$ or $I_\beta$ when we determined $u_i$ or $v_i$ for the very first instance of Step 2.
Hence, the algorithm can also produce the partition \fii{F}{I'_\alpha}{I'_\beta} where  $I'_\alpha=(I_\alpha-S) \cup (I_\beta\cap S)$ and $I'_\beta=(I_\beta-S) \cup (I_\alpha \cap S)$.
If $z_{\ell}\in F$ and $z_{\ell}$ has both an $I_\alpha$-neighbor and an $I_\beta$-neighbor, then, since $z_\ell \in V_3\cup W_4$,  either \fii{F}{I_\alpha}{I_\beta}, \fii{(F-u_1)}{(I_\alpha+u_1)}{I_\beta},
or \fii{(F-u_1)}{I_\alpha}{(I_\beta+u_1)} is an {\FII} of $G$.
Suppose that $z_{\ell}\in F$ and $z_{\ell}$ has no $I_\beta$-neighbor.
Since neither \fii{(F-u_\ell)}{I_\alpha}{(I_\beta+u_\ell)} nor \fii{(F-u_\ell)}{I'_\alpha}{(I'_\beta+u_\ell)} is an {\FII} of $G$,
we have $\{u_{\ell-1}, v_{\ell-1}, u_{\ell-2}\}\cap I_\alpha\neq \emptyset$
and
$\{u_{\ell-1}, v_{\ell-1}, u_{\ell-2}\}\cap I_\beta\neq \emptyset$.
If $u_{\ell-1}\not\in F$, then $v_{\ell-1}\in F$, and so $u_{\ell-2}\not\in F$.
This implies that
either \fii{(F+u_{\ell-1}-u_\ell)}{(I_\alpha-u_{\ell-1})}{(I_\beta+u_\ell)}
or \fii{(F+u_{\ell-1}-u_\ell)}{(I'_\alpha-u_{\ell-1})}{(I'_\beta+u_\ell)} is an {\FII} of $G$.
If $u_{\ell-1}\in F$, then
since neither \fii{F}{I_\alpha}{I_\beta} nor \fii{F}{I'_\alpha}{I_\beta'} is an {\FII} of $G$,
we know $u_2,z_2\in F$.
Now,
either \fii{(F-u_1)}{(I_\alpha+u_1)}{I_\beta}
or \fii{(F-u_1)}{(I'_\alpha+u_1)}{I'_\beta} is an {\FII} of $G$.

Suppose that $z_\ell\not\in F$, and without loss of generality assume  $z_\ell\in I_\alpha$.
Moreover, we may assume that all neighbors of $z_{\ell}$ are in $F$.
Otherwise, it is the case where $z_\ell\in W_4$, which is already covered by the case where $z_{\ell}\in F$ has neighbors in $I_\beta$ and $I_\alpha$.
Since neither \fii{F}{I_\alpha}{I_\beta} nor \fii{F}{I'_\alpha}{I_\beta'} is an {\FII} of $G$,
$u_{\ell-1}\in F$ and $|\{v_{\ell-1}, u_{\ell-2}\}\cap F|\ge 1$.
Now, either \fii{(F-u_\ell)}{I_\alpha}{(I_\beta+u_{\ell})} or
\fii{(F-u_\ell)}{I'_\alpha}{(I'_\beta+u_{\ell})} is an {\FII} of $G$.
\end{proof}

\section{Remarks}\label{sec:remark}

There is a natural generalization of {\FII}s.
For a nonnegative integer $k$,
we say a graph $G$ has an $\mathcal{F}\mathcal{I}_k$-partition $F \sqcup I_{1} \sqcup  \cdots \sqcup I_{k}$
if $F, I_{1}, \ldots, I_{k}$ is a partition of  $V(G)$ such that $G[F]$ is a forest and each $I_{i}$ is a 2-independent set.
As explained in the introduction, a graph with an $\mathcal{F}\mathcal{I}_k$-partition can be star $(k+3)$-colored.
Let $h$ and $f$ be functions such that
\[h(k)=\inf\{\mad(G): G\text{ has no }\mathcal{F}\mathcal{I}_k\text{-partition}\}
\qquad
f(k)=\inf\{\mad(G): \chi_s(G)>k \}.\]
Since a forest is star 3-colorable, for an integer $k$, $ h(k+3)\le f(k)$.

Determining the exact values of $f(k)$ and $h(k)$ is a difficult, yet interesting problem.
From~\cite{2016BrFeKuLoStYa,2009BuCrMoRaWa}, we know  $f(1)=1, f(2)={3\over 2}, f(3)=2$, and  ${5\over 2}\leq f(4)\leq {18\over 7}$.
Our main result implies $f(5)\geq \frac{8}{3}$.
As stated in~\cite{2009BuCrMoRaWa}, determining the exact value of $f(k)$ for $k\ge 4$ remains an intriguing question.

\begin{question}[\cite{2009BuCrMoRaWa}]\rm
What is the exact value of $f(k)$ for $k\ge 4$?
\end{question}

The motivation of $\mathcal{F}\mathcal{I}_k$-partitions comes from star colorings, but it is interesting in its own right.
It is easy to see that a graph $G$ has an
$\mathcal{F}\mathcal{I}_0$-partition if and only if $G$ is a forest.
Since a forest has maximum average degree less than 2, it follows that $h(0)=2$.
Since a graph $H$ with $\mad(H)={5\over 2}$ where $H$ has no $\mathcal{F}\mathcal{I}_1$-partition was constructed in~\cite{2009BuCrMoRaWa}, we know $h(1)\leq {5\over 2}$.
Yet, Brandt et al.~\cite{2016BrFeKuLoStYa} proved that a graph $G$ with $\mad(G)<{5\over 2}$ has an $\mathcal{F}\mathcal{I}_1$-partition, so the value of $h(1)$ is determined, namely, $h(1)=\frac{5}{2}$.
In this term, our main result is equivalent to $h(2)\geq\frac{8}{3}$.
We explicitly ask the question of determining the value of $h(k)$ for $k\geq 2$.

\begin{question}\rm
What is the exact value of $h(k)$ for $k\ge 2$?
\end{question}

It is tempting to guess $h(k)={4+k\over 2}$, yet we provide a construction that shows $h(2)\le \frac{46}{17}<3$.

\begin{construction}\rm
For  a positive integer $n$, let $G_{5n}$ be the graph obtained from  a $5n$-cycle $v_0,\ldots,v_{5n-1}$ by
attaching two pendent triangles to $v_i$ where $i\pmod 5\in\{1,2,3\}$.
It is not hard to see that $\mad(G_{5n})=\frac{46}{17}$.
Now suppose to the contrary that $G_{5n}$ has an {\FII} \fii{F}{I_\alpha}{I_\beta}.
By Lemma~\ref{lema:2B2}, we know that if $i\not\equiv 2\pmod5$ then the vertex $v_i$ of $G_{5n}$ is in $F$. This also forces $v_{5j+2}$ to be in $F$, which is a contradiction since $F$ is a forest.
Hence,  $G_{5n}$ has no {\FII}.
\end{construction}

\begin{figure}[ht]
	\centering
  \includegraphics[scale=0.65,page=8]{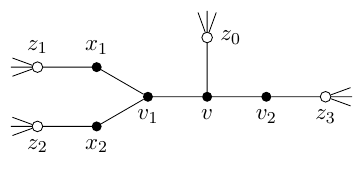} 
\caption{The graphs $G_{5}, G_{10}$, and $G_{15}$.}
  \label{fig:tight}

\end{figure}

As the above infinite family of graphs exhibit $h(2)\le \frac{46}{17}$, we seek the exact value of $h(2)$.
\begin{question}\rm
What is the value $h(2)$? In particular, is $h(2)=\frac{46}{17}$?
\end{question}

As layed out in Table~\ref{table}, a planar graph with girth at least 10 is star 4-colorable~\cite{2016BrFeKuLoStYa}, which is sharp in the sense that the number of colors cannot be reduced~\cite{2004AlChKiKuRa}.
The main result in this paper implies that a planar graph with girth at least 8 is star 5-colorable.
It is also known that there exists a planar graph with girth 7 that is not star 4-colorable~\cite{2008Ti}.
Regarding star 5-colorings, the only remaining case in terms of girth is to determine whether planar graphs with girth 7 are star 5-colorable or not.

\begin{question}
Does there exist a planar graph with girth 8 that is not star 4-colorable or is every planar graph with girth at least 7 star 5-colorable?
\end{question}

\section*{Acknowledgements}
The authors sincerely thank the referees for their valuable comments.
We also thank Xuding Zhu for catching a flaw in the previous manuscript.

 \end{document}